\documentclass{amsart} 
\usepackage{amssymb,amsfonts,amsmath} 
\usepackage[all,2cell]{xy}
\usepackage{graphicx}
\usepackage{pstricks}
\usepackage[mathscr]{eucal}
\usepackage{amscd}
\usepackage{color}		
\usepackage{epsfig}
\usepackage{enumerate}
\usepackage{indentfirst}
\usepackage{fancyhdr}

\theoremstyle{plain}
\newtheorem{thm}{\bf Theorem}[section]
\newtheorem{prop}[thm]{\bf Proposition}
\newtheorem{lem}[thm]{\bf Lemma}
\newtheorem{cor}[thm]{\bf Corollary}

\theoremstyle{definition}
\newtheorem{dfn}[thm]{\bf Definition}
\newtheorem{ex}[thm]{\bf Example}

\theoremstyle{remark}
\newtheorem{rem}[thm]{\bf Remark}

\DeclareMathOperator{\diag}{diag}

\def \Q{\mathbb{Q}}
\def \C{\mathbb{C}}

\def \P{\mathbb{P}}
\def \G{\mathbb{G}_m}
\def \k{\Bbbk}

\newcommand{\acknowledge}{\subsection*{Acknowledgments}}
\newcommand{\conventions}{\subsection*{Notation and Conventions}}

\title[Equivariant Cohomology of Group Embeddings]{Equivariant cohomology of rationally smooth group embeddings}
\author{Richard P. Gonzales}
\address{ 
Mathematisches Institut\\ 
Heinrich-Heine-Universit\"{a}t \\ 
40225 D\"{u}sseldorf\\ 
Germany} 
\email{rgonzalesv@gmail.com}

\begin{document}

\maketitle

\begin{abstract}
We describe the equivariant cohomology ring of rationally smooth 
projective embeddings of 
reductive groups.
These embeddings are the projectivizations of reductive monoids. 
Our main result describes their equivariant cohomology 
in terms of roots, idempotents, and underlying monoid data.
Also, we characterize 
those embeddings 
whose equivariant cohomology ring is obtained 
via restriction to the associated toric variety.  
Such characterization is given in terms of the closed orbits.  
\end{abstract}

\section*{Introduction and motivation}
\addcontentsline{toc}{section}{Introduction}

Let $G$ be a connected complex reductive algebraic group, 
$B\subset G$ a Borel subgroup, and $T\subset B$ a maximal torus.  
A $G$-variety is called {\em spherical} if it is normal 
and contains a dense $B$-orbit. 
A nice feature of spherical $G$-varieties is that they contain only 
finitely many $B$-orbits  
and  
$G$-orbits. 
This yields a complete description of the geometry of 
spherical varieties 
in terms of 
certain combinatorial objects of discrete convex geometry. 
See \cite{ti:sph} and \cite{pe:sph}  
for an up-to-date discussion of spherical varieties  
and a comprehensive bibliography. 
In this paper, we focus on  
a special yet remarkable class of spherical varieties, namely, group embeddings.  

\smallskip

An irreducible  
algebraic variety  
is called an {\em embedding} of $G$, 
or a {\em group embedding}, 
if it is a normal $G\times G$-variety containing an open orbit 
isomorphic to $G$ itself, where $G\times G$ acts on $G$ by left and right multiplication. 
When $G$ is a torus, we get back the notion of toric varieties. 
Group embeddings are 
spherical $G\times G$-varieties 
due to the Bruhat decomposition. 
Affine embeddings of $G$ are nothing but reductive monoids having $G$ as group of units \cite{ri:ge}. 
Recall that an {\em algebraic monoid} is an algebraic variety equipped with an associative 
product map, which is a morphism of varieties and admits an identity element. 
An affine algebraic monoid is called {\em reductive} if it is 
irreducible, normal, and its unit group is a reductive 
algebraic group. 
Reductive monoids have been intensively studied 
in the works of Putcha, Renner, Vinberg, Rittatore, Brion, 
and others. See \cite{pu:lam}, \cite{re:lam}, \cite{bri:monb}, \cite{ti:sph} and the references therein.

\smallskip

Let $M$ be a reductive monoid with zero and unit group $G$.
Then there exists a central one-parameter subgroup 
$\epsilon:\G\to G$ 
such that $\lim_{t\to 0}\epsilon(t)=0$.  
Moreover, 
the quotient space $\P_\epsilon(M):=(M\setminus\{0\})/\epsilon(\G)$
is a normal projective 
embedding of the quotient group $G/\epsilon(\G)$.
These varieties were introduced 
by Renner in his study of algebraic monoids \cite{re:class, re:hpoly, re:ratsm}. 
Projective embeddings of  
connected reductive groups are exactly the 
projectivizations of reductive monoids \cite{re:class}.

\smallskip 

Let $X$ be a complex algebraic variety of dimension $n$. 
Cohomology, in this article, is always considered with rational coefficients. 
We say that $X$ is {\em rationally smooth} 
if we have $H^m(X,X\setminus\{x\})=0$ for $m\neq 2n$, 
and $H^{2n}(X,X\setminus\{x\})=\mathbb{Q}$, for all $x\in X$. 
This is precisely the requirement that $X$ is a {\em rational cohomology manifold}.
See \cite{bri:rat} for a modern account of this key notion. 
Using chiefly methods from the theory of algebraic monoids, 
Renner 
investigated those 
group embeddings that 
are rationally smooth \cite{re:ratsm}, \cite{re:hpolyirr}. 
This class is larger than the class of smooth group embeddings. 

\smallskip 

Now let $X=\P_\epsilon(M)$ be a projective group embedding. 
The purpose of this article is to determine the 
equivariant cohomology rings $H^*_{T\times T}(X)$ and $H^*_{G\times G}(X)$ provided $X$ is rationally smooth.  
Our main results generalize those of Brion \cite{bri:bru} 
and Littelmann-Procesi \cite{lp:equiv} for regular group embeddings.
The main tool in our description is 
the theory developed by Goresky-Kottwitz-MacPherson 
\cite{gkm:eqc}.  
A key result of {\em GKM theory} (Theorem \ref{gkm.thm}) 
gives an explicit presentation  
of the equivariant cohomology of {\em $T$-skeletal} varieties 
(i.e. complete $T$-varieties with only 
finitely many $T$-fixed points and $T$-stable curves)   
whose odd cohomology vanishes. 
If $X$ is rationally smooth, then $X$ has no cohomology in odd degrees,
by a 
previous result of the author  
\cite[Theorem 7.4]{go:cells}. 
Hence, to attain the sought-after descriptions of 
$H^*_{T\times T}(X)$ and $H^*_{G\times G}(X)$,  
we proceed in two steps: first we compute the {\em GKM data} of $X$, i.e. 
$T\times T$-fixed points, $T\times T$-stable curves and the corresponding characters of $T\times T$, 
in terms of the combinatorial data of $M$. 
We remark that the calculation of such data is independent of whether or not $X$ 
is rationally smooth. 
Afterwards, we specialize it to the case when 
$X$ is rationally smooth via Theorem \ref{gkm.thm}.   
Our findings increase the applicability of $GKM$ theory 
in the study of singular group embeddings. 

\smallskip

The results and methods of this paper open the way to further developments, 
e.g., if $X=\P_\epsilon(M)$ is {\em any} group embedding, 
then the GKM data obtained here describes the $T\times T$-equivariant operational 
$K$-theory of $X$, i.e. certain 
ring of piecewise exponential functions acting 
on the $T\times T$-equivariant $K$-theory of $X$, see \cite{go:opk}. 
In addition, a description of the module structure of the 
(rational) equivariant Chow groups of 
rationally smooth group embeddings,   
in the spirit of \cite{go:cells},   
is obtained in \cite{go:rm}. 
Since $X$ is possibly singular, 
and working with intersection theory is usually more 
delicate than with cohomology, the description in \cite{go:rm} 
is not a straightforward generalization of \cite{go:cells}. 
Other techniques are needed. 
These results will appear elsewhere. 

\smallskip 
Here is 
an outline of the paper. 
In the first part (Sections 1 and 2)  
we work over an algebraically closed field $\k$ of  arbitrary characteristic, since  
the 
classification of group embeddings outlined above   
holds in this generality \cite[Chapter 6]{bk:frob}. 
Section 1 gathers some preliminary notions and results from 
the theory of reductive monoids.  
In Section 2 we show that 
any projective group embedding $\P_\epsilon(M)$ is 
$T\times T$-skeletal and assess the corresponding GKM data.  
See Theorem \ref{thefixedpoints.thm}, Theorem \ref{thecurves1.thm} and 
Subsection 2.4.

In the second part (Sections 3 and 4),   
we specialize the results of   
Section 2 to the case when $\k=\C$ and the group embedding 
$X=\P_\epsilon(M)$ is rationally smooth. 
Subsection 3.1 collects basic facts on GKM theory. 
Subsection 3.2 provides 
a short 
discussion  
of rationally smooth group embeddings. 
Subsection 3.3 contains our major results: Theorem \ref{main.thm} gives the ultimate description 
of 
$H^*_{T\times T}(X)$ in terms of 
the finite combinatorial invariants of $M$, i.e. the roots of $(G,T)$ and the Renner monoid. 
Theorem \ref{comparison.thm} 
compares $H^*_{G\times G}(\P_\epsilon(M))$
and $H^*_{T\times T}(\P_\epsilon(\overline{T}))$, where $\P_\epsilon(\overline{T})$ 
is the associated torus embedding.  
We show that, unlike the case of regular embeddings, $H^*_{G\times G}(\P_\epsilon(M))$ 
is in general a proper subring of $H^*_{T\times T}(\P_\epsilon(\overline{T}))^W$, 
where $W$ is the Weyl group of $(G,T)$. 
In Proposition \ref{quasireg.cor} we characterize those embeddings for which these two rings 
coincide; they correspond to the {\em toroidal} rationally smooth group embeddings, i.e.   
those whose closed $G\times G$-orbits are all of the form $G/B\times G/B^-$. 
Section 3 concludes with  
a description of the non-equivariant cohomology of toroidal rationally smooth group embeddings (Theorem \ref{dcptype.cor}). 
Finally, in Section 4, we illustrate the theory 
with a detailed study of {\em simple} projective embeddings, 
i.e. those with only one closed $G\times G$-orbit   
(Theorem \ref{mainj.thm}).

\acknowledge Some of these results were part of the author's doctoral dissertation under the supervision
of Lex Renner. I would like to thank him and Michel Brion for their invaluable help. 
I would also like to thank K\"{u}r\c{s}at Aker, whose interest in the subject allowed me to 
present this work at various universities in Turkey, and 
the Technological Research Council of Turkey (T\"{U}BITAK), 
for the support that I received, as a postdoctoral fellow, 
through Projects 107T897 and 109T667. 
I am also grateful to the German Research Foundation (DFG), Research Grant PE2165/1-1, 
for its support during the final stages of completing this article. 
Lastly, 
I thank the two referees for very helpful comments and suggestions that improved the 
clarity of the article. 

\conventions{ 
We work over an algebraically closed field $\k$. In Sections 1 and 2, $\k$ is of arbitrary characteristic;  
in Sections 3 and 4, $\k=\C$.  
All algebraic varieties and algebraic groups are assumed to be defined over $\k$. 
By a variety we mean a separated reduced scheme of finite type over $\k$; 
in particular varieties 
need not be irreducible. 
A point will always mean a closed point.

$G$ denotes 
a connected reductive 
linear algebraic group 
with Borel subgroup 
$B$ and maximal torus $T\subset B$.   
The Weyl group of $(G,T)$ is denoted by $W$. 
Recall that $W=N_G(T)/T$, where $N_G(T)$ is the normalizer of $T$ in $G$.     
We denote by $\Xi$ the character group of $T$, and by 
$\Phi$ (resp. $\Delta$) the set of roots (resp. simple roots) of $(G,T)$. 
So $\Delta\subset \Phi \subset \Xi$. 
We denote by $s_\alpha\in W$ the reflection corresponding to $\alpha \in \Phi$.  
Observe that $W$ is generated by the simple reflections $\{s_\alpha\}_{\alpha \in \Delta}$.  
We write  
$U_\alpha$ for the unipotent subgroup of $G$ associated to $\alpha \in \Phi$. 

\smallskip

For $w\in W$, denote by ${\rm int}(w)$ the inner automorphism of $T$ given by 
conjugation with (a representative 
of) $w$. 
This yields a $W$-action on $\Xi$ via 
$(w,\chi) \mapsto \chi\circ {\rm int}(w)^{-1}$, where $w\in W$, and $\chi\in \Xi$. 
We denote by $S$ the symmetric algebra over $\Q$ of the abelian group $\Xi$. 
Let $S^W\subset S$ be the subring of $W$-invariants. 
Then 
$S$ is a free $S^W$-module of rank $|W|$. Furthermore, there is a 
graded $W$-stable subspace $R\subset S$, isomorphic to the 
regular representation of $W$, such that $S\simeq R\otimes S^W$ 
as graded $S^W$-modules, see e.g. \cite[Section 3.6]{hum:cox}.

\smallskip 

For a $T$-variety $X$, we denote by $X^T$ the fixed point set, and by $i_T:X^T\to X$ 
the inclusion. 
The (rational) $T$-equivariant cohomology 
of a complex $T$-variety $X$ 
is denoted by $H^*_{T}(X)$. The $T$-equivariant cohomology of a point identifies to $S$, 
where each character has degree 2. More generally, $H^*_T(X)$ is an algebra over $S$. 
}

\section{Preliminaries}

\subsection{Reductive monoids}



We collect a few crucial results  
from the theory of reductive monoids. 
For a complete treatment of the subject, 
the reader is invitated to consult \cite{re:lam} and \cite{pu:lam}.  
Those interested in a survey of the main ideas may also see \cite{so:rm}. 
The semi-expository article \cite{bri:monb} contains more recent developments.  
\smallskip 



\smallskip 
Throughout the paper, $M$ denotes a reductive monoid with zero and unit group $G=G(M)$.  
There is a natural $G\times G$-action on $M$ given by $(g,h)\cdot a:=gah^{-1}$. 
Let $\overline{T}\subset M$ be the Zariski closure of $T$ in $M$.  
It is known that $\overline{T}$ is normal affine toric variety \cite[Theorem 5.4]{re:lam}. 
Moreover, $\overline{T}=\{x\in M\,|\, xt=tx,\; {\rm for\;all\;} t\in T\}$, see \cite[proof of Theorem 5.5]{re:lam}.

\smallskip

We write $E(M)$ for the idempotent set 
of $M$, that is, $E(M):=\{e\in M\,|\, e^2=e\}$. 
Clearly, $E(M)$ is stable under the conjugation action of $G$. 
We denote by $E(\overline{T})$ the idempotent set of $\overline{T}$. 
Observe that $E(\overline{T})$ is invariant under the conjugation action of $W$.  
Define a partial order on $E(M)$ (and thus on $E(\overline{T})\subset E(M)$) 
by declaring 
$f\leq e$ if and only if $fe=f=ef$. 

\medskip

The set $G\backslash M /G$ of $G\times G$-orbits in $M$ is finite (for $M$ is $G\times G$-spherical). 
Moreover, every $G\times G$-orbit contains an idempotent, since $M=E(M)G$ (see \cite[Theorem 4.2]{re:lam}).    
The following properties are specially important in the analysis of the $G\times G$-orbit structure of $M$. 

\begin{itemize} 
\item Any idempotent of $M$ is conjugate 
to one in $\overline{T}$ \cite[Proposition 3.13]{re:lam}. 

\item If $e,f\in E(M)$, then $GeG=GfG$ if and only if $e$ and $f$ are conjugate under $G$ \cite[Proposition 3.13]{re:lam}. 

\item If $e,f\in E(\overline{T})$ are conjugate under $G$, then they are 
conjugate under $W$ \cite[Theorem 6.25]{pu:lam}. 
\end{itemize}
Consequently, 
there are bijections
$$G\backslash M /G \longleftrightarrow E(M)/G \longleftrightarrow E(\overline{T})/W$$
given by
$$ GeG \longleftrightarrow \{geg^{-1} \, | \, g\in G\} \longleftrightarrow \{wew^{-1} \, | \, w\in W\}$$
for $e\in E(\overline{T})$. 
Here 
$E(M)/G$ 
is the set of 
$G$-conjugacy classes in $E(M)$, and 
$E(\overline{T})/W$ 
is the set of 
$W$-conjugacy classes in $E(\overline{T})$.

\smallskip 

We denote by $\Lambda$ {\em the cross section lattice} of $M$ (relative to $T$ and $B$).  
This is the subset of $E(\overline{T})$ 
defined as 
$$\Lambda:=\{e\in E(\overline{T})\,|\,Be=eBe\}.$$ 
It turns out that $\Lambda$ 
can be identified with the (finite) set $G\backslash M /G$ \cite[Theorem 4.5]{re:lam}. 
Therefore, 
$$M=\bigsqcup_{e\in \Lambda}GeG.$$
By our previous remarks, we can also identify $\Lambda$ with
the set $E(\overline{T})/W$. In the sequel, we shall use freely these identifications.

\medskip 
Next we define the Renner monoid.  
Let $R=\overline{N_G(T)}\subset M$. If $x\in R$, 
then  $x=wt$ for some $w\in N_G(T)$ and $t\in \overline{T}$. 
Hence, $xT=Tx$. In fact, one checks that $R=\{x\in M \, | \, Tx=xT\}$, cf. \cite[p. 309]{re:cell}.    
It follows that $\mathcal{R}:=R/T=T\backslash R$ 
has the unique structure of a finite monoid; 
its group of units is $W$ 
and its idempotent set is $E(\overline{T})$. 
Moreover, $\mathcal{R}\simeq E(\overline{T})\cdot W$ 
\cite[Proposition 8.1]{re:lam}. 
The monoid $\mathcal{R}$ is called the {\em Renner monoid} of $M$. 
Observe that in an expression for $x\in \mathcal{R}$ of the form $x=ew$, with $e\in E(\overline{T})$ and $w\in W$, 
the idempotent $e$ is uniquely determined, 
that is, if $x=ew=e'w'$, with $e,e'\in E(\overline{T})$ and $w,w'\in W$, then $e=e'$. 

\medskip 

For $x\in \mathcal{R}$, 
it makes sense to 
talk about the two-sided orbit 
$BxB\subset M$, because  $T\subset B$.  
Remarkably,  
there is an analogue of the Bruhat decomposition for reductive monoids, namely, 
$$M=\bigsqcup_{r\in \mathcal{R}}BrB.$$ 
See \cite[Theorem 8.8]{re:lam} for more details. 

\medskip

On the Renner monoid $\mathcal{R}$ we define the {\em Bruhat-Chevalley order} by  
$$x\leq y \;\;{\rm if \;and\; only\; if}\;\; BxB\subseteq \overline{ByB}.$$ 
The induced poset structure on $W$ coincides with the 
(classical) Bruhat-Chevalley order on $W$. 
This order on $\mathcal{R}$ extends the order on $E(\overline{T})$ 
defined a few paragraphs above. See \cite[Section 8.6]{re:lam} for details.    

\medskip 

The decomposition of $M$ into $G\times G$-orbits has 
its analogue on $\mathcal{R}$, namely,  
$$\mathcal{R}=\bigsqcup_{e\in \Lambda}WeW,$$ 
a decomposition into $W\times W$-orbits. 

\medskip 

We denote by $\mathcal{R}_k$ the set of elements of rank $k$ in $\mathcal{R}$, 
that is, 
$$\mathcal{R}_k=\{x\in \mathcal{R} \, | \, \dim{Tx}=k \,\}.$$ 
Analogously, one defines  
$\Lambda_k\subset \Lambda$ and $E_k\subset E(\overline{T})$. 

\medskip
Finally, we conclude this review by introducing some important subgroups of $G\subset M$.   
For $e\in E(M)$, 
define 
$$
P_e:=\{g\in G\,|\, ge=ege\} 
{\rm \;\; and\;\;}  
P_e^-:=\{g\in G\,|\, eg=ege\}. 
$$
It is known that $P_e$ and $P_e^-$ are opposite parabolic subgroups of $G$, 
with common Levi subgroup $C_G(e)=\{g\in G\,|\, ge=eg\}$, the centralizer of $e$ in $G$. 
In particular $C_G(e)$ is connected and reductive. 
Moreover, the unipotent radical $U_e$ of $P_e$ (resp. $U_e^-$ of $P_e^-$)   
satisfies $U_e\cdot e=\{e\}$ (resp. $e\cdot U_e^-=\{e\}$). See \cite[Theorem 4.5]{re:lam}.  
%
If $e\in E(\overline{T})$, then we write $C_W(e)$ for 
the centralizer of $e$ in $W$. In this case, 
$C_W(e)$ is the Weyl group of $T$ in $C_G(e)$ (see \cite[Section 9.5]{re:cell}). 



\subsection{Projective group embeddings.}


Let $M$ be a reductive monoid with zero and unit group $G$. 
Let 
$\epsilon:\G\to T$ be a central one-parameter subgroup, 
with image $Z$, 
such that $\lim_{t\to 0}\epsilon(t)=0$ \cite[Lemma 1.1.1]{bri:mon}.  
We denote by $\P_\epsilon(M)$ the projective 
group embedding $(M\setminus\{0\})/Z$. 
Recall that all projective embeddings of connected 
reductive groups are 
obtained by this procedure \cite{re:sem}.  
If $M$ is {\em semisimple} (i.e. $G$ has a one-dimensional center), 
then $\epsilon$ is essentially unique, and we write $\P(M)$ 
instead of $\P_\epsilon(M)$.

\begin{ex}\label{construction.ex}
Let $G_0$ be a semisimple algebraic group and 
let $\rho:G_0\to {\rm End}(V)$ be a finite dimensional 
irreducible representation of $G_0$. 
Define $Y_\rho$ to be the Zariski closure of 
$G=[\rho(G_0)]$ in $\mathbb{P}({\rm End}(V))$, 
the projective space associated with ${\rm End}(V)$.
Finally, let $X_\rho$ be the normalization of $Y_\rho$. 
By definition, $X_\rho$ is a projective embedding of $G$. 
Notice that $M_\rho$, the normalization of the Zariski closure of $\k^*\rho(G_0)$ in ${\rm End}(V)$, 
is a semisimple monoid whose group of units is $\k^*\rho(G_0)$.
Embeddings of this kind are studied in more detail in Section 4.
\end{ex}

It follows from Subsection 1.1 that the  
 $G\times G$-orbits in $\P_\epsilon(M)$ are indexed by 
$\Lambda\setminus \{0\}$. 
Similarly, 
the $B\times B$-orbits of $\P_\epsilon(M)$ are indexed by $\mathcal{R}\setminus \{0\}$. 
With these identifications, the set of 
closed $G\times G$-orbits of $\P_\epsilon(M)$ corresponds to $\Lambda_1$. 
Next is a structural description of the $G\times G$-orbits in $\P_\epsilon(M)$.

\begin{prop}\label{orbits.thm}
Let $X=\P_\epsilon(M)$ be a projective group embedding. Let $e\in \Lambda$. 
Let $H_{e}$ denote the $G\times G$-stabilizer of $[e]\in X$.     
Then 
there is a fibration sequence
$$\xymatrix{eC_G(e)/\G \ar@{^(->}[r]& (G\times G)/H_{e} \ar[r]^{\pi \; \;}& G/P_e\times G/P_e^-}.$$
In particular, if $e\in \Lambda_1$, then
$$(G\times G)/H_{e}\simeq G/P_e\times G/P_e^-,$$
for, in this case, $eMe\simeq \k$, $eC_G(e)\simeq e\times \G$ and $P_e\cdot e=\G \cdot e$.   
\end{prop}

\begin{proof}
Note that $H_e$ 
is contained in the subgroup $P_e\times P_e^-$. 
To see this, let $(g,h) \in H_e$. 
Then $geh^{-1}=ez$, for some $z\in Z\simeq \G$. 
That is, $egeh^{-1}z^{-1}=e^2$, but $e$ is an idempotent, 
so $egeh^{-1}z^{-1}=e$. The latter yields $ege=ezh$, and the 
right hand side equals $ge$, by assumption. 
Thus $ege=ge$. Analogously, one checks $eh=ehe$. 
Since $H_e\subset P_e\times P_e^-$, there is a natural map 
of homogeneous spaces $\pi: (G\times G)/H_e\to G/P_e\times G/P_e^-$, whose fiber 
is $(P_e\times P_e^-)/H_e$. 
We claim that this fibre is isomorphic to $eC_G(e)/eZ\simeq eC_G(e)/\G$.   
Indeed, first recall that $P_e=C_G(e)U_e$ and $P_e^-=C_G(e)U_e^-$. 
Moreover, $U_e\cdot e=\{e\}$ 
and $e\cdot U_e^-=\{e\}$ (Subsection 1.1). 
Also, by \cite[Theorem 4.8 (a)]{re:lam}, $eC_G(e)$ is a connected reductive group with unit $e$. 
Hence the map 
$p_e:P_e\times P_e^-\to (C_G(e)e/Ze)\times (eC_G(e)/eZ)$, sending $(g,h)$ to $([ge],[eh])$ 
is a well-defined surjective group homomorphism (cf. \cite[Proof of Proposition 2.2]{pe:mon}). 
%
By considering the corresponding morphism of Lie algebras, 
one easily checks  
that the differential of $p_e$ at $(1,1)$ 
is surjective; that is, $p_e$ is separable \cite[Theorem 4.3.7 (iii)]{sp:lag}.   
Now observe that $p_e$ maps $H_e$ onto ${\rm diag}(eC_G(e)/eZ)$. 
In consequence, $P_e\times P_e^-/H_e\simeq (eC_G(e)/eZ)\times (eC_G(e)/eZ)/{\rm diag} (eC_G(e)/eZ)\simeq eC_G(e)/eZ$.    
For the last assertion of the Proposition, 
notice that $e\in \Lambda_1$ is a minimal idempotent, so $eMe$ is a reductive 
monoid isomorphic to $\k$, with unit group $eC_G(e)\simeq \G$    
\cite[Theorem 4.8]{re:lam}. 
\end{proof}


Finally, {\em associated} to $\P_\epsilon(M)$, there is a
projective torus embedding of $T/Z$, namely, 
$\P_\epsilon(\overline{T})=[\overline{T}\setminus \{0\}]/Z.$ 
By construction, $\P_\epsilon(\overline{T})$ 
is a (normal) projective toric variety contained in $\P_\epsilon(M)$. 
Notice that the $T$-orbit structure of $\P_\epsilon(\overline{T})$ is governed by $E(\overline{T})\setminus \{0\}$.
This toric variety 
will play an important role in Sections 3 and 4. 



\section{$GKM$ data of projective group embeddings} 

%
We maintain the notation from Section 1. 
Let $X=\P_\epsilon (M)$ be a projective group embedding.   
%
%
In this section we show that  
$X$ is $T\times T$-skeletal, i.e. 
$X$ has only finitely many $T\times T$-fixed points and $T\times T$-invariant curves. 
Furthermore, 
for each $T\times T$-invariant curve of $X$ 
we obtain explicitly 
the associated character of $T\times T$. 
We write down this {\em GKM data} 
in terms of the combinatorial invariants of $M$. 
The calculations 
do not depend on any special property of $M$ or $X$. 

\smallskip
Our initial task is to identify the following two sets.
\begin{enumerate}
\item $\{x\in M\;|\; \dim TxT=1\}$.
\item $\{x\in M\;|\; \dim TxT=2\}$.
\end{enumerate}
The first class will determine the set $X^{T\times T}$, 
whereas the second one will determine the set of $T\times T$-invariant curves in $X$.

\subsection{Fixed Points}

We identify $\mathcal{R}_1$ with its image in $X=\P_\epsilon(M)$, 
and consider it 
as a subset of $X$.   
Also, $Z\subseteq T$ is the given attractive one-parameter subgroup in the center of $G$.

\begin{thm}  \label{thefixedpoints.thm}
The subsets $\mathcal{R}_1$ and $X^{T\times T}$ are equal.  
In particular, there is only a finite number of $T\times T$-fixed points in $X$.
\end{thm}
\begin{proof}
The set $X^{T\times T}$ corresponds to
$\{x\in M\;|\; \dim(TxT)=1\}$.   
Note that if $\dim(Tx)=1$, then $Tx=Zx$. Similarly, if $\dim(xT)=1$,
then $xT=Zx$.
These remarks, together with the fact that $Tx\cup xT\subseteq TxT$, yield the equality
$$\{x\in M\;|\; \dim(TxT)=1\}
= \{x\in M\;|\; Tx=xT\;\text{and}\;\dim(Tx)=1\}.$$
The latter set is precisely $\mathcal{R}_1$. 
\end{proof}

\subsection{Invariant Curves}

\begin{prop} \label{tricho.prop}
Let $x\in M$ and assume that $x\neq 0$. Then the following are equivalent.
\begin{enumerate}
\item $\dim TxT=2$.
\item Either $\dim(xT)=2$ and $Tx\subseteq xT$, $xT=TxT$; or 
      $\dim(Tx)=2$ and $xT\subseteq Tx$, $Tx=TxT$; or $\dim(TxT)=2$ and $Tx=xT=TxT$.
\end{enumerate}
\end{prop}
\begin{proof}
It suffices to check that (1) implies (2), for the other direction is obvious. 
So assume that (1) holds. Now $Tx\cup xT\subseteq TxT$. If $\dim(Tx)=\dim(xT)=1$, then
$Tx=Zx=xT$. 
But then $\dim(TxT)=1$, a contradiction. Hence at least one of $Tx$ or $xT$
is two-dimensional. If $\dim(Tx)=2$, then $Tx\subseteq TxT$ yet they have the
same dimension. Thus $Tx=TxT$. If $\dim(xT)=2$, then we end up with $xT=TxT$. 
\end{proof}

\begin{cor}   \label{tricho.cor}
Exactly one of the following assertions is true for $x\in M$ such that $\dim(TxT)=2$.
\begin{enumerate}
\item $xT\subset Tx=TxT$ and $\dim(xT)=1$.
\item $Tx\subset xT=TxT$ and $\dim(Tx)=1$.
\item $xT=Tx=TxT$. \hfill $\square$
\end{enumerate}
\end{cor}


The following result, due to Renner \cite[Lemma 3.3]{re:hpoly}, will be needed in the sequel. 
%

\begin{lem}   \label{monoidbb2.lem}
Let $M$ be a reductive monoid with zero and unit group $G$.
Let $T\subseteq G$ be a maximal torus.
Choose a central one-parameter subgroup $\epsilon:\G \to G$, with image $Z$, that converges to $0$.
Then
\[
\{x\in M\backslash\{0\}\;|\; Zx=Tx\}=\bigsqcup_{e\in E_1}eG.
\]
Consequently, if $X=\P_\epsilon(M)$ and $eX=(eM\backslash\{0\})/Z\simeq eG/Z$
then
\[
X^T=\bigsqcup_{e\in E_1}eX
\]
for the action $T\times X\to X$ given by $(t,[x])\leadsto [tx]$. Similar
results hold for the right action $([x],t)\leadsto [xt]$ of $T$ on $X$. \hfill $\square$
\end{lem}


\begin{thm}      \label{thecurves1.thm} 
There are only three types of closed irreducible $T\times T$-curves in $X$, namely: 
\begin{enumerate}
\item $\overline{U_\alpha [ew]}$, where $e\in E_1$, $s_\alpha\notin C_W(e)$ and $w\in W$ (fixed pointwise by $T$ on the right).
\item $\overline{[we]U_\alpha}$, where $e\in E_1$, $s_\alpha\notin C_W(e)$ and $w\in W$ (fixed pointwise by $T$ on the left).
\item $\overline{T[x]}=\overline{[x]T}=\overline{T[x]T}$, where $x\in\mathcal{R}_2=\{x\in\mathcal{R}\;|\;\dim(Tx)=2\}$.
\end{enumerate}
In particular, the number of $T\times T$-invariant curves in $X$ is finite. 
\end{thm}
\begin{proof}
Keeping the numeration of Corollary \ref{tricho.cor}, 
we know that the $T\times T-$curves of $X$ fall into three classes.
The first two types correspond, as Lemma \ref{monoidbb2.lem} dictates, to curves that are fixed pointwise
by $T$ on either the left or the right. The former collection lies on 
$X^T=\bigsqcup_{e\in E_1}eG/Z$. Moreover,
due to the Bruhat decomposition, for each $e\in E_1$ the following identity holds
$$eG/Z=G/{P_e^-}=\bigsqcup_{r\in eW}[r]B_u,$$ where $B_u$ is the unipotent radical of $B$.

Our task is to find all the $T$-curves of $eG/Z$, where $e$ varies over all the rank-one 
idempotents of $\overline{T}$.
So fix an idempotent $e\in E_1$. 
It follows from the results of Carrell (\cite{c:schu}, \cite[Lemma 2.2]{ck:curves}), 
that the $T$-curves of $eG/Z$ are of the form
$[r]U_\alpha$, for some root $\alpha$ 
such that $s_\alpha \notin C_W(f)$ and $f=w^{-1}ew$ (here, $r=ew=wf$).  
Indeed, since $f$ is a rank-one idempotent,
then $s_\alpha\in C_W(f)$ if and only if $U_\alpha f=fU_\alpha=\{f\}$ \cite[Lemma 5.1]{re:cell}.
Because there is no essential difference
between $e$ and $f$, we conclude that a
$T\times T$-curve, $TxT$, is fixed pointwise on the left by $T$ if and only if $TxT=wfU_\alpha$,   
where $\alpha \notin C_W(f)$, $f\in E_1$, and $w\in W$. 
A similar argument disposes of the case when a $T\times T$-curve is fixed pointwise by $T$ on the right.

Finally, if $Tx=xT=TxT$ and ${\rm dim}(Tx)=2$, then $x\in \mathcal{R}_2$. Identifying $x\in \mathcal{R}_2$
with its image $[x]$ in $X$, it is clear that $T[x]T$ is a $T\times T$-curve in $X$.
%
\end{proof}

Let us state Theorem \ref{thefixedpoints.thm} and Theorem \ref{thecurves1.thm} in a more compact form.

\begin{thm}\label{ratsmisgeneric.thm}
Let $X=\P_\epsilon(M)$ be a projective group embedding. Then,  
for its natural $T\times T$-action,  
$X$ 
is $T\times T$-skeletal. \qed   
\end{thm}


\subsection{Classification of $GKM$-curves}
Let $X=\mathbb{P}_\epsilon(M)$ be a projective group embedding.  
From Theorem \ref{thecurves1.thm}, we also know that there 
are three types of $T\times T$-curves in $X$:
\begin{enumerate}
\item Curves that are fixed pointwise by $T$ on the right: 
      $\overline{U_\alpha [ew]}$, $e\in E_1$, $s_\alpha\notin C_W(e)$, and $w\in W$.
\item Curves that are fixed pointwise by $T$ on the left: 
      $\overline{[we]U_\alpha}$, $e\in E_1$, $s_\alpha\notin C_W(e)$,  and $w\in W$. 
\item $\overline{T[x]}=\overline{[x]T}=\overline{T[x]T}$ where $x\in\mathcal{R}_2=\{x\in\mathcal{R}\;|\;\dim(Tx)=2\}$.
\end{enumerate}

But which pair of fixed points, i.e. elements of
$\mathcal{R}_1$, is joined by each of these curves?
Preserving the given order, we obtain
\begin{enumerate}
\item   $ew$ and $s_\alpha ew$
\item  $we$ and $wes_\alpha$
\item  The two elements $r,s\in\mathcal{R}_1$ such that $r,s\in \overline{TxT}$. 
\end{enumerate}

\begin{thm}      \label{thecurves2.thm}
The set of $T\times T$ - curves in $X=\mathbb{P}_\epsilon(M)$ is identified
as follows, by pairs of $T\times T$-fixed points. Here $Ref(W)$ refers to the set of
reflections of $W$, the ordering on $\mathcal{R}$ is the Bruhat-Chevalley order (Section 1.1). 
\begin{enumerate}
\item $\{(x,sx)\;|\; x\in\mathcal{R}_1, s\in Ref(W)\;\text{and}\;x > sx\}$.
\item $\{(x,xs)\;|\; x\in\mathcal{R}_1, s\in Ref(W)\;\text{and}\;x > xs\}$.
\item $\mathcal{R}_2\cong\{A\subseteq\mathcal{R}_1\;|\; |A|=2\;\text{and}\;
A=\{ex,fx\}\;\text{for some}\;e,f\in E_1\;\text{and some}\;
x\in\mathcal{R}_2\}$.
\end{enumerate}
\end{thm}
\begin{proof}
Assertions (1) and (2)  
follow from the fact that if $x\neq sx$ and $s\in Ref(W)$, then either $x < sx$ or
else $sx < x$. Indeed, first write $x$ in normal form \cite[Definition 8.34]{re:lam}, 
that is, write $x=veu$, where $e\in \Lambda_1$ and $v,u$ are 
minimal length coset representatives in $W/C_W(e)$,    
then use \cite[Corollary 8.35]{re:lam}.    
For (3) we proceed as follows. Recall that any $x\in \mathcal{R}_2$
can be written as $x=fu$, where $f\in E_2$ is a rank-two idempotent, and $u\in W$.
Since $u$ is invertible, it is enough to prove the statement for $x=f$. Now notice that 
$(f\overline{T}\setminus \{0\})/Z$ is isomorphic to $\P^1$ (see e.g. \cite[Corollary 1.4.1]{bri:rat}).  
Thus there are exactly two
fixed points, they correspond to the unique rank-one idempotents $e,e' \in E_1$
such that $ef\neq 0$ and $e'f\neq 0$. These two idempotents determine $f$ uniquely (to see this, 
simply represent $\overline{T}$ as a closed submonoid of ${\rm End}(\k^n)$ 
consisting of diagonal matrices). Finally, note that the pair $(e,e')$ cannot be any of the 
ones indexed in type (1) or (2) above, for otherwise either $e'=se$ or $e'=es$, where $s$ is a reflection and $s\notin C_W(e)$. 
But then $se \in E_1$ or $es\in E_1$, and any of these would imply $se=es=e$, 
by the uniqueness of the decomposition in the Renner monoid, 
a contradiction. 
\end{proof}

Notice that the description in  (3) above is just a convenient, indirect way of identifying
the elements of $\mathcal{R}_2$ as pairs of $T\times T$ - fixed points. Notice also that, for each
$x\in \mathcal{R}_2$, there are exactly two idempotents $e,f\in E_1$ such that $ex\neq 0$
and $fx\neq 0$.

\smallskip

Any $T\times T$-fixed point of $X$ is contained in a closed $G\times G$-orbit (indeed, recall that  
$\mathcal{R}_1=\bigsqcup_{e\in \Lambda_1}WeW$, and for each $e\in \Lambda_1$, the orbit $WeW$  
is identifiable with the $T\times T$-fixed points of the complete homogeneous space $G[e]G$).   
The curves identified
in (1) and (2) of Theorem \ref{thecurves2.thm} are the ones that are contained in closed
$G\times G$-orbits. The curves identified in (3) of Theorem \ref{thecurves2.thm} are those
that are not contained in any closed $G\times G$-orbit. 
As in \cite{bri:bru}, these curves are further separated 
into whether or not the corresponding fixed points are in the same
closed $G\times G$-orbit (Lemma \ref{thetaforex3.lem}).  
This distinction will become relevant in the next section when we
identify the character associated with each $T\times T$-curve of type (3). 

\begin{ex}   \label{thecurvesofmn.ex}
We illustrate Theorem \ref{thecurves2.thm} with the example $M=M_n(\k)$. Let $E_{i,j}$ denote an
elementary matrix. We then obtain (with the ordering as in Theorem \ref{thecurves2.thm})
\begin{enumerate}
\item $\{(E_{i,j},E_{k,j})\;|\; i\neq k\}$.
\item $\{(E_{i,j},E_{i,k})\;|\; j\neq k\}$.
\item $\{(E_{i,j},E_{k,l})\;|\; i\neq k\;\text{and}\; j\neq l\}$.
\end{enumerate}
In each case the associated curve is the $T\times T$-orbit of the sum of the
given pair of elementary matrices. In case (1) the two elementary matrices are in the
same row. In case (2) the two elementary matrices are in the same column. Case (3) 
determines the remaining cases.
\end{ex}

\subsection{The Associated Characters} 
Let us briefly recall how a character is 
associated with a $T$-stable curve. 
Let $X$ be a complete $T$-variety and let 
$C$ be a $T$-stable irreducible curve of $X$, 
which is not fixed pointwise by $T$. 
Let $\pi:\tilde{C}\to C$ be the ($T$-equivariant) normalization.  
Then $\tilde{C}$ is isomorphic to $\P^1$. 
Denote by $0,\infty$ the two fixed points 
of $T$ in $\tilde{C}$, and denote by $x_0,x_\infty$ 
their corresponding images 
via $\pi$. 
Then 
$\tilde{C}\setminus \{0,\infty\}=C\setminus \{x_0,x_\infty\}$ 
identifies to $\k^*$, where $T$ acts on 
$\tilde{C}\setminus \{0,\infty\}$ 
via a unique character $\chi$ 
(when interchanging $0$ and $\infty$, 
one replaces $\chi$ by $\chi^{-1}$). 

\smallskip

In this Subsection we identify the character $\theta_x=(\lambda_x,\rho_x)$ of $T\times T$ associated
with a $T\times T$-curve $\overline{T[x]T}$ in $X=\P_\epsilon(M)$. 
As discussed previously
(Theorems \ref{thecurves1.thm} and \ref{thecurves2.thm}), there are three different types of $T\times T$-curves in $X$.
For the curves of type (1) and (2) 
the whole issue reduces to the well-documented situation of \cite{c:schu} and \cite[Lemma 2.2]{ck:curves},  
since these curves are contained in the closed $G\times G$-orbits of $X$ (such closed orbits are complete homogeneous spaces, Proposition \ref{orbits.thm}). Hence,    
\begin{enumerate}
\item For a $T\times T$-curve $\overline{U_\alpha [ew]}$, with $e\in E_1$, $s_\alpha\notin C_W(e)$, and $w\in W$, 
the associated character is $(\alpha,1)$. 
\item For a $T\times T$-curve $\overline{[we]U_\alpha}$, with $e\in E_1$, $s_\alpha\notin C_W(e)$,  and $w\in W$, 
the associated character is $(1,\alpha)$.  
\end{enumerate}
      
Consequently, we only need to focus on the curves of type (3), that is, those of the form $\overline{T[x]T}$, with $x\in\mathcal{R}_2$. 

\smallskip

So let $x\in\mathcal{R}_2$. Since we are working on the monoid level,
the initial step in our discussion is to calculate the map
\[
m_x : T\times T \to TxT,\;\;(s,t)\leadsto sxt^{-1}.
\]
We then compose $m_x$ with the canonical map $\pi_x : TxT\to TxT/Z\cong \G$ to obtain
\[
\theta_x=\pi_x\circ m_x
\]
where $Z\subseteq G$ is the given central, attractive, 1-parameter
subgroup of the unit group $G$ of $M$. Notice that
$\theta_x$ depends on the choice of group isomorphism $TxT/Z\cong \G$. The other
isomorphism $TxT/Z\cong \G$ yields $\theta_x^{-1}$. In the calculation of
$\theta_x$ it is important to keep
track of this ambiguity. It is also useful to consider the map
\[
\tau_x : T\to Tx, t\leadsto tx
\]
and the character $\lambda_x=\pi_x\circ \tau_x$. Notice that $TxT=Tx$, so we wish to express
$\theta_x : T\times T \to \G$ as a composition
\[
T\times T \to T\times T\to T\to Tx\to \G
\]
involving the $W\times W$-action on $T\times T$, the multiplication $T\times T\to T$, and these other
quantities: $\tau_x$, $\pi_x$, $\lambda_x$.
We also assess how the $W\times W$-action on $\mathcal{R}_2$ affects     
the characters associated to the curves $\overline{T[x]T}$, $x\in \mathcal{R}_2$.  
This will effectively reduce the calculation of 
$\theta_x$, with  $x\in\mathcal{R}_2$, 
to calculating $\theta_x$ for a set of representatives
of the $W\times W$-orbits of $\mathcal{R}_2$.

\smallskip

%
Write $x$ as $x=fu=ug$, where $u\in W$ and $f,g\in E_2$. 
An elementary calculation yields that
\[
m_x : T\times T \to TxT=xT,\;\;(s,t)\leadsto sxt^{-1}
\]
is given by $m_x(s,t)=s(t^u)^{-1}x$ where, by definition, $t^u=utu^{-1}$. Recall that
$\lambda_x=\pi_x\circ \tau_x$, where $\tau_x : T\to Tx, t\leadsto tx$, and $\pi_x : TxT\to TxT/Z\cong \G$.

\begin{lem} \label{thetaforex1.lem}
Write $\theta_x = (\lambda_x,\rho_x)\in \Xi\oplus \Xi$, where $\Xi\oplus \Xi$ is the character group of $T\times T$.  
Then
\begin{enumerate}
  \item $\lambda_x=\lambda_f$.
  \item $\rho_x=\lambda_g^{-1}=\lambda_f^{-1}\circ int(u)$, where $int(u)(t)=utu^{-1}$.
\end{enumerate}
\end{lem}
\begin{proof}
Consider $m : T\times T \to Tf,\;\;(s,t)\leadsto s(t^u)^{-1}f$. Then $m(s,t)\in Zf$
if and only if $m_x(s,t)\in Zx$. So $ker(\pi_f\circ m)=ker(\pi_x\circ m_x)$.
Thus $\lambda_x=\lambda_f$ and $\rho_x=\lambda_f^{-1}\circ int(u)$. 
On the other hand, $m$ is also the product of $(s,1)\leadsto sf$ and $(1,t)\leadsto (t^u)^{-1}f$.
The first of these is $\lambda_f$ and the second of these is
$\lambda_f^{-1}\circ int(u)$. Clearly, $(t^u)^{-1}f\in Zf$ if and only if $t^{-1}g\in Zg$, 
because $ugu^{-1}=f$. Thus $ker(\lambda_x^{-1}\circ int(u))=ker(\lambda_g^{-1})$. We conclude that
$\theta_x=(\lambda_x,\rho_x)=(\lambda_f,\lambda_g^{-1})=(\lambda_f,\lambda_f^{-1}\circ
int(u))$. 

It is worth noting that we can also write $m_x$ as $m_x : T\times T \to TxT=xT$, \; 
$m_x(s,t)=sxt^{-1}=xs^{u^{-1}}t^{-1}$ (notice that $x$ now appears on the left). 
The resulting calculation then yields 
$\theta_x=(\lambda_x,\rho_x)=(\lambda_g\circ int(u^{-1}),\lambda_g^{-1})=(\lambda_f,\lambda_g^{-1})$. 
\end{proof}

Observe that either $\theta_x=(\lambda_x,\lambda_x^{-1}\circ int(u))$ or
$\theta_x=(\lambda_x^{-1},\lambda_x\circ int(u))$ depending on the orientation.
Also, it follows from Lemma \ref{thetaforex1.lem} that if $f,g \in E_2$ are conjugate, 
i.e. $f=wgw^{-1}$ for some $w\in W$, then $\lambda_f=\lambda_g\circ int (w^{-1})$.

\begin{lem} \label{thetaforex2.lem}
Let $x\in\mathcal{R}_2$. 
Let $\theta_x=(\lambda_x,\rho_x)$ be the associated character, as in Lemma \ref{thetaforex1.lem}. 
\begin{enumerate}
  \item If $y=xw$, where $w\in W$, 
  then $\theta_y=(\lambda_x,\rho_x\circ int(w)).$
  \item If $y=wx$, where $w\in W$, 
  then $\theta_y=(\lambda_x\circ int(w^{-1}),\rho_x)$. 
\end{enumerate}
\end{lem}
\begin{proof} 
As before, write $x=fu$, with $f\in E_2$ and $u\in W$.  
If $y=xw$, then $y=fuw$, and so Lemma \ref{thetaforex1.lem}, 
yields $\theta_y=(\lambda_f,\lambda_f^{-1}\circ int(uw))=(\lambda_f,\lambda_f^{-1}\circ int(u)\circ int(w))$. 
Thus $\theta_y=(\lambda_x,\rho_x\circ int(w))$.    
On the other hand, if $y=wx$, then $y=wfu=f'wu$, where $f'=wfw^{-1}$.  
Once again, Lemma \ref{thetaforex1.lem} implies $\theta_y=(\lambda_{f'},\lambda_{f'}^{-1}\circ int(wu))$. 
Moreover, $\lambda_{f'}=\lambda_f\circ int (w^{-1})$.  
Thus $\theta_y=(\lambda_x\circ int(w^{-1}),\rho_x)$. 
\end{proof}

Next we state a slight variation on a result of 
Putcha \cite[Proposition 10.9]{pu:lam}. 

\begin{lem}\label{hclass.lem}
Let $f\in E_2$. Then either there is a unique $s\in C_W(f)$ such that 
$sf=fs\neq f$, or every $s\in C_W(f)$ satisfies $sf=fs=f$. 
In the former case, $s$ is a reflection, denoted $s_{\alpha_f}$. 
Moreover, $\lambda_f=\alpha_f$, a root of $(G,T)$. 
\end{lem}

\begin{proof}
Since $C_W(f)=\langle s_\alpha \;|\; \alpha \in \Delta,\; s_\alpha f=fs_\alpha \rangle$ 
\cite[Lemma 10.15]{pu:lam}, 
the first two assertions follow directly from \cite[Proposition 10.9]{pu:lam}. 
Thus, assuming there is a unique reflection 
$s:=s_{\alpha_f}\in C_W(f)$ with the property that $sf=fs\neq f$, 
it remains to show that $\lambda_f=\alpha_f$. 
In fact, since 
the map $T\to Tf$ is $s$-equivariant and 
$s$ acts trivially on ${\rm Ker}(T\to Tf)$ \cite[Corollary 10.11]{pu:lam}, 
it suffices to check that $\lambda_f$ and $\alpha_f$ agree on $Tf$.  
For this,  
consider the inner transformation $int(s):f\overline{T}\to f\overline{T}$, $fx\mapsto fsxs^{-1}$. 
Let us examine the automorphism $\sigma$ induced by $int(s)$ on 
$f\overline{T}-\{0\}/Z\simeq \P^1$. 
Recall that there are exactly two rank-one idempotents $f_1$ and $f_2$ below $f$.
Denote by $0$ and $\infty$, respectively, their classes in the orbit space 
$f\overline{T}-\{0\}/Z$. Also, since $f$ is the identity element of the 
reductive monoid $f\overline{T}$, let us denote its class on $\P^1$ by $1$.
Because $(sf_is^{-1})\cdot f=sf_is^{-1}$ for $i=1,2$, it is clear that
$\sigma$ permutes the points $0$ and $\infty$. So either 
$\sigma(0)=0$ and $\sigma(\infty)=\infty$ 
or else $\sigma(0)=\infty$ and $\sigma(\infty)=0$.
Moreover, $\sigma(1)=1$ in either case, because 
$\sigma$ restricts to an algebraic automorphism of $\G\simeq Tf/Z=\P^1\setminus\{0,\infty\}$.
Hence, as a M\"{o}bius transformation, 
$\sigma$ is either $z\mapsto z$ or $z\mapsto z^{-1}$.
The former is clearly impossible because, by assumption, $sf=fs\neq f$. 
Therefore, by looking
at the commutative diagram
$$\xymatrix{
Tf     \ar[rr]^{int(s)} \ar[d]^{\pi} &  & Tf \ar[d]_{\pi} \\
Tf/Z   \ar[rr]^{z\mapsto z^{-1}}      &  & Tf/Z  \\}
$$ 
we conclude that $s$, when restricted to $Tf$, is also a reflection. 
Clearly, $\alpha_f$ is uniquely determined by the commutative diagram above, 
thus $\lambda_f=\alpha_f$. 
\end{proof}

%
%

\begin{lem} \label{thetaforex3.lem}
Let $x\in \mathcal{R}_2$, and write $x=fu$, 
with $f\in E_2$ and $u\in W$. 
The following are equivalent. 
\begin{enumerate}
\item There is a unique reflection $s_{\alpha_f}$ such that $s_{\alpha_f}f=fs_{\alpha_f}\neq f$. 
\item The two $T\times T$-fixed points of $\overline{T[x]T}\subset X$  
are in the same $W\times W$-orbit. 
\end{enumerate}
\end{lem}
\begin{proof}
Let $x\in\mathcal{R}_2$ and let $a,b\in\overline{TxT}$ 
represent the two $T\times T$-fixed points 
in $\overline{T[x]T}$. 
Then  
$a=f_1x=f_1u$ and $b=f_2x=f_2u$ where $f_1,f_2$ are the two rank-one
idempotents below $f$. 
Assume (1) holds. 
Set $s:=s_{\alpha_f}$.  
Notice that $sf_1s=f_2$ since $sf=fs\neq f$. 
So by setting $t:=u^{-1}su$, one gets $b=sat$ and $a=sbt$. 

Now let $x=fu\in\mathcal{R}_2$ and assume that $f_1x=f_1u$ and $f_2x=f_2u$ are in the same
$W\times W$-orbit. Then $f_1$ and $f_2$ are in the same $W\times W$-orbit. 
This implies that $f_1$ and $f_2$ are conjugate (Subsection 1.1). 
Furthermore, \cite[Corollary 8.9 and Proposition 10.9]{pu:lam} asserts that $f_1$ and $f_2$ are conjugate
by an element $s\in C_W(f)=\{v\in W\;|\; vf=fv\}$. 
It follows that $sf=fs\neq f$ (for otherwise, $sf_1s^{-1}=f_1$ but $f_1\neq f_2$). 
Thus $s=s_{\alpha_f}$ (Lemma \ref{hclass.lem}).  
\end{proof}

The next result follows immediately from Lemmas \ref{thetaforex1.lem}, \ref{thetaforex2.lem} and \ref{hclass.lem}.  

\begin{lem} \label{thetaforex4.lem}
Let $x\in\mathcal{R}_2$. 
Write $x=fu$ and assume there is a (unique) reflection $s_{\alpha_f}$
such that $s_{\alpha_f}f=fs_{\alpha_f}\neq f$. 
Then $\lambda_f=\alpha_f$ and $\lambda_f\circ {\rm int}(s_{\alpha_f})=\lambda_f^{-1}$. 
%
In particular, the components of $\theta_x=(\lambda_x,\rho_x)$ 
are roots of $(G,T)$, namely,   
$$\theta_x=(\lambda_x,\rho_x)=(\alpha_f, \alpha_f^{-1} \circ int(u)). $$ 
Moreover, if $y=s_{\alpha_f}x$, then 
$\theta_y=(\lambda_x^{-1},\rho_x)=(\alpha_f^{-1},\alpha_f^{-1} \circ int(u))$. \qed
\end{lem}
%
%

\begin{ex} \label{curvesformn.ex}
Let $M=M_n(\k)$ and let $T$ be the set of invertible, diagonal matrices. One checks that
\[
\mathcal{R}_2=\{E_{i,j} + E_{k,l}\;|\; i\neq k\;\text{and}\; j\neq l\}.
\]
where $E_{i,j}$ denotes the elementary matrix with a one in the $(i,j)$-position 
and zeros elsewhere. Let $\underline{s}=(s_1,...,s_n)\in T$
denote the obvious diagonal matrix. A simple calculation yields that, for $\underline{s},
\underline{t}\in T$ and $x=E_{i,j} + E_{k,l}$,
\[
\theta_x(\underline{s},\underline{t})=s_i^{-1}s_kt_jt_l^{-1}.
\]
In this case, $x=fu$, with $f=E_{i,i}+E_{k,k}$ and $u=\sigma_{j,i}\sigma_{k,l}$, where  
$\sigma_{i,j}$ represents the permutation matrix that exchanges 
the $i$-th and $j$-th columns of the identity matrix.  
The unique reflection $s_f$ satisfying $s_ff=fs_f\neq f$ is $\sigma_{i,k}$. 
The corresponding $y=s_f x\in \mathcal{R}_2$ from Lemma \ref{thetaforex4.lem} is 
$y=E_{k,j} + E_{i,l}$, 
Thus,
\[
\theta_y(\underline{s},\underline{t})=s_is_k^{-1}t_jt_l^{-1}.
\]
In the terminology of Lemma \ref{thetaforex4.lem}, $\theta_x = (\lambda_x,\rho_x)$
where $\lambda_x=\alpha_{k,i}=s_ks_i^{-1}$ and $\rho_x=\alpha_{j,l}=t_jt_l^{-1}$. 
Similarly,
$\lambda_y=\alpha_{i,k}$ and $\rho_y=\alpha_{j,l}$. 
\end{ex}


\section{$GKM$ theory of rationally smooth projective group embeddings}
%
In this and next section 
we specialize the results of Section 2 to the case when $\k=\C$ and 
the group embedding $\P_\epsilon(M)$ is rationally smooth. 

From now on, 
in order to consider $S$, the symmetric algebra over $\Q$ of the character group of $T$,  
we shall write characters in the additive notation, that is, $(\chi_1+\chi_2)(t):=\chi_1(t)\chi_2(t)$.

\subsection{GKM theory} 

We recall some basic facts and nomenclature from \cite{gkm:eqc}. 
A $T$-variety $X$ 
is called 
{\em equivariantly formal} if 
the natural map $H^*_T(X)\to H^*(X)$ is surjective (i.e.  
$H^*(X)$ is the quotient of $H^*_T(X)$ 
by its homogeneous ideal generated by all characters of $T$).  
If $X^T$ is finite, then $X$ is equivariantly formal if and only if 
$X$ has no cohomology in odd degrees \cite[Lemma 1.2]{bri:eu}.  

A $T$-variety is a {\em GKM variety} if it is complete, $T$-skeletal, and equivariantly formal. 

\begin{thm}[\protect{\cite[Theorem 1.2.2]{gkm:eqc}}] \label{gkm.thm}
Let $X$ be a GKM variety, and let $X^T=\{x_1,\ldots,x_m\}$. 
Then the pullback  
$$i_T^*:H^*_T(X)\to H^*_T(X^T)$$ 
is injective, and its image is the set of all 
$(f_1,\ldots,f_m)\in S^m$ such that $f_i\equiv f_j \mod \chi$ whenever $x_i,x_j$ are 
connected by an irreducible invariant curve where $T$ 
acts through the character $\chi$. \qed 
\end{thm}
For notational purposes, we shall interpret the image of $i^*_T$ as the set of all 
maps $\varphi:X^T\to S$ such that $\varphi(x_i)\equiv \varphi(x_j) \mod \chi$ whenever 
$x_i$ and $x_j$ are connected by an irreducible invariant curve where $T$ 
acts through the character $\chi$. 


\smallskip 

By a result of Borel \cite{bo:sem}, if $X$ is a $G$-variety, 
then $H^*_G(X)\simeq H^*_T(X)^W$. 

\begin{lem}\label{comparison.formality}
Let $X$ be a $G$-variety. Suppose that $X$ has no cohomology in odd degrees and that, 
for the induced $T$-action, $X^T$ is finite. Then 
$H^*_G(X)$ and $H^*_T(X)$ are free modules over $S^W$ and $S$ 
respectively, and their ranks satisfy 
$${\rm rank}_{S^W}H^*_G(X)={\rm dim}_\Q{H^*(X)}={\rm rank}_{S}H^*_T(X)=|X^T|.$$ 
\end{lem}

\begin{proof}
By hypothesis, $X$ is equivariantly formal for the induced $T$-action, 
so $H^*_T(X)$ is a free $S$-module \cite[Lemma 1.2]{bri:eu}. 
Now let $F$ be a graded $W$-stable complement of $S_+\cdot H^*_T(X)\subset H^*_T(X)$, 
where $S_+$ is the ideal in $S$ consisting of polynomials without constant term. 
One checks that 
$H^*_T(X)\simeq S\otimes F$. 
It is known that  
$S\simeq S^W\otimes R$, where $R$ is isomorphic to the regular representation of $W$. 
So $H^*_T(X)\simeq S^W\otimes R\otimes F$, which implies $H^*_G(X)\simeq S^W\otimes (R\otimes F)^W$. 
But then $(R\otimes F)^W\simeq F$, because $R$ is the regular representation of $W$.   
Thus,  
${\rm rank}_{S^W}H^*_G(X)={\rm rank}_{S}H^*_T(X)={\rm dim}_\Q{H^*(X)}=\dim F$. 
Finally, by the Localization Theorem (see e.g. \cite[Lemma 1.1]{bri:eu}) 
we conclude that  
${\rm rank}_{S}H^*_T(X)=|X^T|.$ 
\end{proof}

\subsection{Rationally smooth group embeddings} 
Recall that for $e\in E(M)$, the monoid    
$M_e:=\overline{\{g\in G\,|\,ge=eg=e\}}$  
is a  
reductive monoid with 
$e$ as its zero element \cite[Corollary 2.3.3]{bri:mon}. 
Next is Renner's characterization of rationally smooth group embeddings. 

\begin{thm}[\protect{\cite{re:ratsm}}] \label{ratsm.thm}
Let $\mathbb{P}_\epsilon(M)$ be a projective group embedding. Then the following are equivalent.
\begin{enumerate}
\item $\mathbb{P}_\epsilon(M)$ is rationally smooth.
\item $M\setminus\{0\}$ is rationally smooth.
\item For any minimal, nonzero, idempotent $e$ of $M$, $M_e$ is rationally smooth.
\item For any maximal torus $T$ of $G$, $\overline{T}\setminus\{0\}$ is rationally smooth. \hfill $\square$
\end{enumerate}
\end{thm} 
In particular, $\P_\epsilon(M)$ is rationally smooth 
if and only if $\P_\epsilon(\overline{T})$ 
is rationally smooth 
(a toric variety is 
rationally smooth if and only if it is simplicial \cite{da:tor}).   
Also, notice that the condition does not depend on the choice of $Z$.


\begin{ex} 
Smooth group embeddings are clearly rationally smooth. 
In particular, so are the projective 
{\em regular} embeddings: 
smooth projective $G$-embeddings 
whose closed $G\times G$-orbits are all of the form $G/B\times G/B^-$ (cf. Proposition \ref{quasireg.cor}).  
\end{ex}

\begin{ex}\label{ssrank2.ex}
Let $M$ be a semisimple monoid with zero and unit group $G$ 
of the form $\C^*\times G_0$, where $G_0$ is a simple 
algebraic group of type $A_2$, $C_2$ or $G_2$. 
Then $\P(M)$ is {\em always} rationally smooth.  
Indeed, this follows from Theorem \ref{ratsm.thm}, since,  
in this context, $\P(\overline{T})$ is a simplicial toric surface. 
Note that there are cases when $\P(M)$ has closed $G\times G$-orbits 
of the form $G/P\times G/P^-$, where $P\supsetneq B$ 
is a parabolic subgroup.  
Such embeddings, though rationally smooth, are not regular. See \cite[Section 5.1]{re:lam} for details.  
\end{ex}

\begin{ex}\label{simple.emb.ex}
Let $G$ be a semisimple adjoint group with 
Borel subgroup $B$ and maximal torus $T\subset B$. 
An embedding of $G$ is called 
{\em simple} if it contains a unique closed $G\times G$-orbit. 
Let $X$ be such an embedding. Then, using the notation from Example \ref{construction.ex}, 
$X$ is of the form $\P(M_{\rho_\lambda})$, for some irreducible 
representation $\rho_\lambda$ of $G$, 
with highest weight $\lambda$ \cite{re:ratsm}.    
Moreover, the unique closed $G\times G$-orbit of $X$ 
is the partial flag variety 
$G/P_J\times G/{P_J^-}$, 
where $J=\{\alpha\in \Delta \;|\; s_\alpha(\lambda)=\lambda\}.$ 
Recall that $P_J$ is the standard parabolic subgroup associated to $J$, and 
$P_J^-$ its opposite. 
Using Theorem \ref{ratsm.thm}, Renner has classified  
all rationally smooth simple embeddings combinatorially, 
in terms of $J$ and the Dynkin diagram for $G$ \cite{re:hpolyirr}. 
In Section 4 we present Renner's list of all possible 
$J$'s that yield rationally smooth group embeddings, and discuss the 
connections to Timashev's description of projective group embeddings  
via weight polytopes \cite[Section 27]{ti:sph}.  
\end{ex}


Next is a   
justification for the use of GKM theory 
in the study of rationally smooth group embeddings.  
It is a consequence of Theorem \ref{ratsmisgeneric.thm} 
and \cite[Theorem 7.4]{go:cells}.  

\begin{thm}\label{ratsmimpliesgkm.thm}
Let 
$X=\mathbb{P}_{\epsilon}(M)$
be a projective group embedding. 
If $X$ is rationally smooth, then, for its natural $T\times T$-action,  
$X$ is a GKM variety. \qed 
\end{thm}


\subsection{The main results.} 

Let $X=\P_\epsilon(M)$ be a projective group embedding. 
Recall that the closed $G\times G$-orbits of $X$ correspond to idempotents $e\in \Lambda_1$.
Because the $T\times T$-fixed points of $G[e]G\simeq G/P_e\times G/P_e^-$ are identifiable with the two-sided orbit $WeW$, 
equivariant cohomology 
classes in $H^*_{T\times T}(G[e]G)$ correspond, via Theorem \ref{gkm.thm}, 
to functions $\varphi_e:WeW\to S\otimes S$ 
satisfying the conditions: 
\begin{enumerate}[(i)]
\item $\varphi_e(ew)\equiv \varphi_e(s_\alpha ew) \mod (\alpha,0)$  
whenever $s_\alpha \notin C_W(e)$ and $w\in W$,  
\smallskip 

\item $\varphi_e(we)\equiv \varphi_e(wes_\alpha) \mod (0,\alpha)$ 
whenever $s_\alpha \notin C_W(e)$ and $w\in W$.  
\end{enumerate} 

Now we state the first major result of this article. 
For the analogous result
in the case of projective regular embeddings, 
see \cite[Theorem 3.1.1]{bri:bru}.

\begin{thm} \label{main.thm}
Let $X=\mathbb{P}_\epsilon(M)$ be a projective group embedding. If $X$ is rationally smooth, then the natural map 
      $$H^*_{T\times T}(X) \longrightarrow H^*_{T\times T}\left( \bigsqcup_{e\in \Lambda_1} G[e]G \right)=\bigoplus_{e\in \Lambda_1}H^*_{T\times T}(G[e]G)$$ 
      is injective. In fact,  
       its image consists of all tuples $(\varphi_e)_{e\in \Lambda_1}$, indexed over $\Lambda_1$ and with 
       $\varphi_e \in H^*_{T\times T}(G[e]G)$, 
       subject to the additional conditions:
   \smallskip
   \begin{enumerate} 
          \item If $f\in E_2$ and there is a (necessarily unique) reflection $s_{\alpha_f}$ satisfying  
          $s_{\alpha_f} f=fs_{\alpha_f}\neq f$, then
                $$\varphi_{e_f}(f_1 u)\equiv \varphi_{e_f}(f_2 u) \, \, {\rm mod} \, (\alpha_f, -\alpha_f \circ {\rm int}(u)),$$
                for all $u\in W$. 
                Here,   
                $f_1$ and $f_2=s_{\alpha_f}\cdot f_1 \cdot s_{\alpha_f}$ are the two idempotents in $E_1$ below $f$, 
                the root $\alpha_f$ corresponds to the reflection $s_{\alpha_f}$,
                and 
                $e_f \in \Lambda_1$ is the unique element of $\Lambda_1$ which is conjugate to $f_1$.  
                
\smallskip
           \item If $f\in E_2$ and $sf=fs=f$ for every $s\in C_W(f)$, then 
                 $$\varphi_{e_1}(f_1 u)\equiv \varphi_{e_2}(f_2 u) \, \, {\rm mod} \, (\lambda_f,-\lambda_f \circ int(u)),$$
                 for all $u\in W$. Here,  
                 $\lambda_f$ is the character of $T$ defined by the composition 
                 $$T\to Tf\to Tf/\G\simeq \G, $$  
                 the idempotents $f_1, f_2$ are the unique idempotents below $f$,
                 and $e_i\in \Lambda_1$ is conjugate to $f_i$, for $i=1,2$. 
    \end{enumerate}        
\end{thm}

\begin{proof}
Since $X$ is a GKM variety (Theorem \ref{ratsmimpliesgkm.thm}) and 
$X^{T\times T}\subset \bigsqcup_{e\in \Lambda_1}G[e]G$, one easily checks that 
the natural map 
$H^*_{T\times T}(X) \longrightarrow H^*_{T\times T}\left( \bigsqcup_{e\in \Lambda_1} G[e]G \right)$
is injective. 
%
Now we apply Theorem \ref{gkm.thm}  
to describe the image.  
First, observe that the curves of type (1) and (2) in Theorem \ref{thecurves1.thm} 
are contained in $\bigsqcup_{e\in \Lambda_1}G[e]G$, and 
these curves describe 
$H^*_{T\times T}\left( \bigsqcup_{e\in \Lambda_1} G[e]G \right)$ (see e.g. \cite[Proposition 6.5]{bri:eqchow}).  
%
%
Thus, to conclude the proof, we just need to show that the curves of type (3) in Theorem \ref{thecurves1.thm} 
yield assertions (a) and (b). 
So let $\overline{[TxT]}$, with $x=fu\in \mathcal{R}_2$, be one of these curves. 
By Lemma \ref{hclass.lem}, either there exists a unique reflection $s_{\alpha_f}$ such that 
$s_{\alpha_f}f=fs_{\alpha_f}\neq f$, or $sf=fs=f$ for all  $s\in C_W(f)$.  
In the first case, 
Lemma \ref{thetaforex3.lem} implies that the two fixed points of $\overline{[TxT]}$, namely $f_1x$ and $f_2x$, 
lie in the same closed $G\times G$-orbit (here recall that $f_1,f_2$ are the two idempotents below $f$).  
Moreover, $f_2$ is conjugate to $f_1$ via $s_{\alpha_f}$, namely, 
$f_2=s_{\alpha_f} \cdot f_1 \cdot s_{\alpha_f}$. 
We now use Lemma \ref{thetaforex4.lem} to write the associated character
$\theta_x$ as $\theta_x=(\alpha_f,-\alpha_f \circ int(u))$ (in additive notation),  
where $\alpha_f$ is the root associated to the reflection $s_{\alpha_f}$. 
Since $\Lambda_1$
indexes all closed $G\times G$-orbits in $X$, 
there exists a unique $e_f\in \Lambda_1$ such that $f_1$ and $e_f$ are conjugate.  
Assertion (a) is now proved.
Finally, if $sf=fs=f$ for all $s\in C_W(f)$, 
then $f_1$ and $f_2$ are not conjugate (Lemma \ref{thetaforex3.lem}).  
That is, $f_1x$ and $f_2x$ lie in different closed $G\times G$-orbits. 
Since $x=fu$, Lemma \ref{thetaforex1.lem}  
finishes the proof.   
\end{proof}

The previous result provides a complete combinatorial description
of the equivariant cohomology of any rationally smooth projective group embedding. 
Furthermore, since $X$ is a GKM-variety, 
the non-equivariant cohomology 
$H^*(X)$ can be recovered from Theorem \ref{main.thm}
via $H^*(X)\simeq H^*_{T\times T}(X)\otimes_{S\otimes S}\Q$. 
As pointed out, Brion \cite[Theorem 3.1.1]{bri:bru} 
has obtained a result analogous to Theorem \ref{main.thm} 
for projective regular embeddings. 
Since the latter is a subclass of  
the class of rationally smooth projective group embeddings,   
our Theorem \ref{main.thm} 
implies \cite[Theorem 3.1.1]{bri:bru}.  
%


\medskip

The $G\times G$-equivariant cohomology of $X$ is obtained by means of 
the formula $H_{G\times G}^*(X)\simeq (H^*_{T\times T}(X))^{W\times W}.$

\begin{cor}\label{gequiv.thm}
If $X=\P_\epsilon(M)$ is a rationally smooth group embedding, then  
the ring $H_{G\times G}^*(X)$ consists of 
all tuples $(\psi_e)_{e\in \Lambda_1}$, where $\psi_e\in (S \otimes S)^{C_W(e)\times C_W(e)}$,  
satisfying the following conditions: 
\begin{enumerate}[(a)]
 \item If $f \in \Lambda_2$ and there is a (unique) reflection $s_{\alpha_f}$ such that  
          $s_{\alpha_f} f=fs_{\alpha_f}\neq f$, then  
       $$(s_{\alpha_f},s_{\alpha_f})\,\psi_e \equiv \psi_e \, \,{\rm mod}\, (\alpha_f,-\alpha_f),$$ 
       where $e\leq f$ and the root $\alpha_f$ corresponds to the reflection $s_{\alpha_f}$. 
       
 \item If $f\in \Lambda_2$ and $sf=fs=f$ for every reflection $s\in C_W(f)$, then 
       $$\psi_{e}\equiv \psi_{e'} \, \, {\rm mod}\, (\lambda_f,-\lambda_f),$$
       where $e,e'\leq f$, and $\lambda_f$ is the character of $T$ defined by $f$. 
 \end{enumerate}
\end{cor}

\begin{proof}
Note that $W\times W$ acts on a tuple $(f_r)$ in $H^*_{T\times T}(\mathcal{R}_1)=\displaystyle \bigoplus_{r\in \mathcal{R}_1}S\otimes S$ via 
$(u,v)\cdot (f_r):=((u,v)\cdot f_{u^{-1}\,r\,v}).$ 
For a subgroup $H$ of $G$, we denote by $EH\to BH$ the  
universal principal $H$-bundle. 
Now 
let $e\in \Lambda_1$. 
From \cite[p. 25]{bri:ech} it follows that $H^*(B P_e)\simeq H^*(B C_G(e))\simeq H^*(BT)^{C_W(e)}$, 
because $C_G(e)$ is the Levi subgroup of $P_e$. 
Consequently, 
$$
\begin{array}{ccc}
H^*_{G\times G}(G[e]G) & \simeq & H^*(B P_e)\otimes H^*(B P_e^-)\\
                       & \simeq & H^*(B C_G(e))\otimes H^*(B C_G(e))\\
                       & \simeq & {H^*(BT)}^{C_W(e)}\otimes H^*(BT)^{C_W(e)}\\
                       & \simeq     & (S \otimes S)^{C_W(e)\times C_W(e)}. 
\end{array}
$$ 
Moreover, one checks that this isomorphism is induced by restriction 
to the fixed point $[e]$ of $G[e]G$ (cf. \cite[Section 6.6]{bri:eqchow}).  
By $W\times W$-invariance, the restriction of $\varphi_e \in H^*_{G\times G}(G[e]G)$ 
to the fixed point $(u,v)\cdot e=u^{-1}ev$ is equal to $(u,v)\cdot \psi_e$,  
where $\psi_e=\varphi_e(e)$. 
Thus the relations (1)
and (2) of Theorem \ref{main.thm} reduce to the proposed descriptions (a) and (b).  
\end{proof}

\medskip

Let $Y=\P_\epsilon(\overline{T})\subset X$ 
be the associated toric variety of $X$ (Subsection 1.2).    
Our next theorem allows to compare the equivariant cohomologies of 
$X$ 
and $Y$. 
The situation for rationally smooth group embeddings contrasts deeply with 
the corresponding one for regular embeddings, cf.  
\cite[Corollary 3.1.2]{bri:bru} and \cite[Corollary 2.2.3]{uma:kth}. 
It is worth noting 
that 
the idea of comparing 
$Y$ and $X$ 
goes back to \cite{lp:equiv}.

\begin{thm}\label{comparison.thm}
If $X=\P_\epsilon(M)$ is rationally smooth, then  
the inclusion of the associated torus embedding $\iota:Y\hookrightarrow X$ 
induces an injection:
$$
\xymatrix{\iota^*:H_{G\times G}^*(X) \ar@{^(->}[r] & H^*_{T\times T}(Y)^W\simeq (H^*_{T}(Y)\otimes S)^W,}
$$
where the $W$-action on $H^*_{T\times T}(Y)$ is induced from the action of {\rm diag}$(W)$ on $Y$. 
Furthermore, $\iota^*$ is an isomorphism if and only if $C_W(e)=\{1\}$ for every $e\in \Lambda_1$. 
\end{thm}

\begin{proof}
Since $X$ is rationally smooth, then $Y$ is rationally smooth as well (Theorem \ref{ratsm.thm}). 
Therefore, we have the following commutative diagram
$$
\xymatrix{
H^*_{T\times T}(X)  \ar@{^(->}[r]           \ar[d]^{\iota^*} &  H^*_{T\times T}(X^{T\times T}) \ar[d]^{\iota^*} \\
H^*_{T\times T}(Y) \ar@{^(->}[r]              & H^*_{T\times T}({Y}^{T\times T})_,                 \\
}
$$ 
where the horizontal maps are injective, because both group embeddings are GKM. 
On the other hand, 
recall that $\Lambda_1$ provides a set of representatives of 
both the $W\times W$-orbits in $X^{T\times T}=\mathcal{R}_1$
and the $W$-orbits 
in $Y^{T\times T}=E_1$. 
Thus, after taking 
invariants, we obtain an injection
$$H^*_{T\times T}(\mathcal{R}_1)^{W\times W}=\bigoplus_{e\in \Lambda_1}(S\otimes S)^{C_W(e)\times C_W(e)} 
\hookrightarrow H^*_{T\times T}(E_1)^W=\bigoplus_{e\in \Lambda_1}(S\otimes S)^{C_W(e)}.$$
Placing this information into the commutative diagram above
shows that the restriction map 
$$\iota^*: H^*_{T\times T}(X)^{W\times W} \longrightarrow H^*_{T\times T}(Y)^W$$ 
is injective.
Now observe that $H^*_{T\times T}(Y)^W\simeq (H^*_{T}(Y)\otimes S)^W$.
Indeed, we have a split exact sequence
$$
\xymatrix{
1 \ar[r]& {\rm diag}(T) \ar[r]& T\times T \ar[rrr]^{(t_1,t_2)\mapsto t_1t_2^{-1}}& & & T \ar@/^1pc/ @{-->}[lll] \ar[r]& 1,
}
$$
where the splitting is given by $t\mapsto (t, 1)$.
It follows that $T\times T$ is canonically isomorphic to ${\rm diag}(T)\times(T\times {1})$. 
Furthermore, 
by definition,
${\rm diag}(T)$ acts trivially on $Y$. 
As a consequence, we have a
ring isomorphism $H^*_{T\times T}(Y)\simeq S\otimes H^*_T(Y)$. 
This isomorphism is further $W$-invariant since the $W$-action on the cohomology rings is induced from
the action of $\diag(W)$ on $Y$.

\smallskip

To prove the second part of the Theorem, we adapt to our situation
an argument of Littelmann and Procesi \cite[Theorem 2.3]{lp:equiv}.

First, assuming that $\iota^*$ is also surjective, 
we need to show that $C_W(e)=\{1\}$ for all $e\in \Lambda_1$.
Since $X$ is equivariantly formal, then $H^*_{G\times G}(X)$ is 
a free $(S\otimes S)^{W\times W}$-module of rank $|\mathcal{R}_1|$ (Lemma \ref{comparison.formality}).
And $H^*_{T\times T}(Y)$ is a free $S\otimes S$-module of rank $|E_1|$, for the same reason.
Recall that we can choose a graded $W\times W$-submodule $R$ of
$S\otimes S$, isomorphic to the regular representation of $W\times W$, such that 
$$
S\otimes S\simeq R\otimes (S\otimes S)^{W\times W}
$$
as graded $(S\otimes S)^{W\times W}$-module. 
Accordingly, $H^*_{T\times T}(Y)^{W}$
is in a natural way a free $(S\otimes S)^{W\times W}$-module.
%
Since, by assumption, $\iota^*$ is a graded isomorphism of free $(S\otimes S)^{W\times W}$-modules, we
conclude that the ranks of $H^*_{G\times G}(X)$ and $H^*_{T\times T}(Y)^W$ must be the same.
The next step consists in finding out a more intrisic formula
for the rank of the latter module,  
so as to compare it with $|\mathcal{R}_1|$. 

Since $H^*_{T\times T}(Y)$ is a free $S\otimes S$-module, 
we can find a graded $W$-stable submodule $U$ 
of $H^*_{T\times T}(Y)$
such that the morphism $$U\otimes (S\otimes S)\longrightarrow H^*_{T\times T}(Y)$$
is a $W$-equivariant isomorphism of graded $S\otimes S$-modules (cf. proof of Lemma \ref{comparison.formality}). Here $\dim U=rank_{S\otimes S}H^*_{T\times T}(Y)=|E_1|$.  
Hence 
$$H^*_{T\times T}(Y)^W\simeq (U\otimes S\otimes S)^W\simeq (U\otimes R\otimes (S\otimes S)^{W\times W})^W\simeq (U\otimes R)^W\otimes (S\otimes S)^{W\times W}.$$ 
Since $R$ decomposes into
the direct sum of $|W|$-copies of the regular representation of $W$, then Lemma \ref{aux.lem}  
shows that $\dim{(U\otimes R)^W}=|E_1||W|$. 
Consequently,
$H^*_{T\times T}(Y)^W$ as a free 
$(S\otimes S)^{W\times W}$-module of rank $|E_1||W|$. 
In summary, the surjectivity of $\iota^*$ implies that $|\mathcal{R}_1|=|E_1||W|$.
Now Lemma \ref{aux.idem.lem} finally yields $C_W(e)=\{1\}$, for all $e\in \Lambda_1$.

\medskip

For the converse, suppose that $C_W(e)=\{1\}$ for all $e\in \Lambda_1$. 
We need to show that $\iota^*$ is surjective. To achieve our goal, we modify
slightly an argument of \cite[Section 4.1]{lp:equiv} and \cite[Corollary 3.1.2]{bri:bru}.
Define the variety 
$\mathcal{N}=\bigcup_{w\in W} wY.$
We claim that this union is, in fact, a disjoint union.
Indeed, observe that $\mathcal{N}$ contains all the $T\times T$-fixed points of $X$.
That is, $\mathcal{N}$ has $|\mathcal{R}_1|$ fixed points.
On the other hand, each $wY$ has $|E_1|$ fixed points (for its corresponding $T$-action). 
Now, if it were the case that 
there is a pair of distinct subvarieties $wY$ and $w'Y$
with non-empty intersection, then 
this intersection should also contain $T\times T$-fixed points.
But then a simple counting argument would yield $|\mathcal{R}_1|<|E_1||W|$. 
This is impossible, by our assumptions and Lemma \ref{aux.idem.lem}.  
Hence,  
$\mathcal{N}=\bigsqcup_{w\in W} wY.$

\smallskip

Clearly, $\mathcal{N}$ is rationally smooth and equivariantly formal 
(because each $wY$ is so, for $w\in W$). 
Since $\mathcal{N}$ contains all the $T\times T$-fixed points of $X$,
the restriction map 
$$H^*_{T\times T}(X)\to H^*_{T\times T}(\mathcal{N})$$
is injective. 
Moreover, by Theorem \ref{thecurves2.thm},  
$\mathcal{N}$ contains all $T\times T$-curves of $X$ 
which are not in a closed $G\times G$-orbit, i.e.  
those which contribute to relations (1) and (2) of Theorem \ref{main.thm}.  
As a consequence, 
restriction to $\mathcal{N}$ induces isomorphisms 
$$H^*_{T\times T}(X)^{W\times W}\simeq H^*_{T\times T}(\mathcal{N})^{W\times W}\simeq 
\left(\bigoplus_{w\in W} H^*_{T\times T}(Y)\right)^{W\times W}
\simeq H^*_{T\times T}(Y)^W.$$
The proof is now complete.  
\end{proof}


\begin{lem}[\cite{lp:equiv}]\label{aux.lem}
If $N$ is a finite group, and $U$ and $V$ are two finite dimensional representations
of $N$ such that $V$ is the sum of copies of the regular representation of $N$,then 
$$\dim{(V\otimes U)}^N=\frac{\dim{V}\cdot \dim{U}}{|N|}.$$ \hfill $\square$ 
\end{lem}


\begin{lem}\label{aux.idem.lem}
Let $\mathcal{R}_1$ be the set of rank one elements of the Renner monoid $\mathcal{R}$.
Then $|\mathcal{R}_1|=|E_1|\cdot|W|$ if and only if $C_W(e)=1$ for every $e\in \Lambda_1$.  
\end{lem}

\begin{proof}
Recall that 
$\Lambda_1$ can be identified with a set of representatives of the $W\times W$-orbits
in $\mathcal{R}_1$. Likewise, $\Lambda_1$ also corresponds to a set of representatives of the $W$-orbits
in $E_1$. Let $k$ be the cardinality of $\Lambda_1$ and let $e_1,\ldots, e_k$ be a complete
list of the elements of $\Lambda_1$. Since we are dealing with elements of rank one, 
we have $We_iW\simeq (W/C_W(e_i))\times (W/C_W(e_i))$, for all $i=1,\ldots,k$.  
Indeed, consider the maps $\zeta_i:We_iW\to (W/C_W(e_i))\times (W/C_W(e_i))$ given by 
$\zeta_i(\sigma e_i\tau^{-1})=(\sigma e_i \sigma^{-1}, \tau e_i \tau^{-1})$. 
By the uniqueness of the decomposition in $\mathcal{R}$,  
the various $\zeta_i$ are well-defined and surjective. 
To check that they are also injective, 
note that $e_i\sigma=e_i\sigma=e_i$ for any $\sigma\in C_W(e_i)$, because $e_i$ is minimal. 
Thus
$$|\mathcal{R}_1|=\sum_i|We_iW|=\sum_i |W/C_W(e_i)|^2.$$
On the other hand, the orbit $We_i\subset E_1$ satisfies $We_i\simeq W/C_W(e_i)$.
This implies the following formula 
$$|E_1|=\sum_i|We_i|=\sum_i|W/C_W(e_i)|.$$
The result now follows. 
\end{proof}

\smallskip

Next we characterize those group embeddings for which the map $\iota^*$ of Theorem \ref{comparison.thm} is an isomorphism. 

\begin{prop}\label{quasireg.cor}
Let $X=\P_\epsilon(M)$ be a projective group embedding. 
Let $Y=\P_\epsilon(\overline{T})$ be the associated torus embedding and 
$\iota:Y\to X$ the canonical inclusion. 
Then following are equivalent: 
\begin{enumerate}[(a)]
 \item $C_W(e)=\{1\}$ for every $e\in E_1$.
 \item All closed $G\times G$-orbits in $X$ are isomorphic to $G/B\times G/{B^-}$.
\end{enumerate}

\noindent           If, moreover, $X$ is rationally smooth, then (a) and (b) are equivalent to 
\smallskip 

\noindent (c) The induced map $\iota^*:H^*_{G\times G}(X)\to H^*_{T\times T}(Y)^W\simeq (H^*_T(Y)\otimes S)^W$
is an isomorphism. 
\end{prop}

\begin{proof}
For the equivalence between (a) and (b) remember that every closed $G\times G$-orbit in $X$ is
of the form $G/P_e\times G/{P_e^-}$, for $e\in E_1$. Also, recall that $C_G(e)$, the common
Levi subgroup of $P_e$ and $P_e^-$, has Weyl group equal to $C_W(e)$. 
Then $C_W(e)=\{1\}$, for all $e\in E_1$,
if and only if $P_e$ and $P_e^-$ are two mutually opposite Borel 
subgroups containing $T$, for all $e\in E_1$.
If, in addition, $X$ is rationally smooth, then 
the equivalence between statements (c) and (a) follows at once from Theorem \ref{comparison.thm}, 
since $\Lambda_1$ is the set of representatives of the $W$-orbits in $E_1$.  
\end{proof}

Group embeddings satisfying the equivalent conditions (a) and (b) 
of Proposition \ref{quasireg.cor} are called {\em toroidal} group embeddings,  
see \cite[Section 29]{ti:sph}. 
Observe that smooth toroidal group embeddings are exactly the regular embeddings 
of reductive groups \cite[Theorem 29.2]{ti:sph}. 
Our Theorem \ref{comparison.thm} 
states that the results of 
\cite{lp:equiv}, \cite{uma:kth} and
\cite{bri:bru} can be extended to the class of {\em toroidal rationally smooth} projective group embeddings.
Furthermore, Theorem \ref{comparison.thm} 
gives a precise relation between our results 
and those on simplicial toric varieties. 
Indeed, if $X=\P_\epsilon(M)$ is a toroidal rationally smooth group embedding, then 
$H^*_{G\times G}(X)$ is isomorphic to the subring of $W$-invariants in 
$H^*_T(Y)\otimes S$, where $H^*_T(Y)$ is the ring of piecewise polynomial functions 
on the fan of $Y=\P_\epsilon(\overline{T})$, see \cite{bv:kth}. 

\smallskip

We conclude this section 
describing the non-equivariant cohomology ring of toroidal rationally smooth 
group embeddings. This result is known for regular embeddings, by 
work of  
De Concini-Procesi \cite{dp:reg} 
and Littelmann-Procesi \cite{lp:equiv}. 

\begin{thm}\label{dcptype.cor}
Let $M$ be a reductive monoid with zero and unit group $G$.  
Let $K$ be a maximal compact subgroup of $G$ such that $T_K=T\cap K$ is a
maximal compact torus. 
Suppose that the projective group embedding $X=\P_\epsilon(M)$ is toroidal and rationally smooth. 
Then 
$$
H^*(X)\simeq H^*((K\times K)\times_{(T_K\times T_K)}Y)^W,
$$ 
where $Y=\P_\epsilon(\overline{T})\subset X$ is the associated toric variety.
\end{thm}

\begin{proof}
As $G/K$ is contractible, the functors $H^*_{G\times G}(-)$ and $H^*_{K\times K}(-)$
agree on $G\times G$-spaces. Similar remarks apply to $H^*_{T\times T}(-)$ and $H^*_{T_K\times T_K}(-)$, for $T/T_K$
is also contractible.
Since $X$ is toroidal and rationally smooth, Theorem \ref{comparison.thm} yields 
$$H^*(X)\simeq H^*_{T_K\times T_K}(Y)^W/\mathcal{I}H^*_{T_K\times T_K}(Y)^W,$$
where $\mathcal{I}$ is the ideal of $(S\otimes S)^{W\times W}$ 
generated by the elements of strictly positive degree.
%
As pointed out in \cite[Remark 2.3]{lp:equiv}, 
the induction formula \cite[p. 552]{q:spec} implies that
$$H^*_{T_K\times T_K}(Y)=H^*_{K\times K}((K\times K)\times_{T_K\times T_K}Y).$$
and the latter is isomorphic to
$$(S\otimes S)^{W\times W}\otimes_\Q H^*((K\times K)\times_{T_K\times T_K}Y),$$
because $(K\times K)\times_{T_K\times T_K}Y$ has no cohomology in odd degrees. 
Hence, $$H^*_{T_K\times T_K}(Y)^W\simeq (S\otimes S)^{W\times W}\otimes_\Q H^*((K\times K)\times_{T_K\times T_K}Y)^W,$$ 
and thus  
$$H^*_{T_K\times T_K}(Y)^W/\mathcal{I}H^*_{T_K\times T_K}(Y)^W
\simeq
H^*((K\times K)\times_{T_K\times T_K}Y)^W.  
$$ 
We conclude that  
$H^*(X)\simeq H^*((K\times K)\times_{T_K\times T_K}Y)^W.$  
\end{proof}


\section{Simple group embeddings} 

Let $H$ be a connected reductive group. Recall that an embedding of $H$ is called
simple if it contains only one closed $H\times H$-orbit. 

\subsection{Rationally smooth simple group embeddings. Classification.}

A reductive monoid $M$ with $0$ is called {\em $\mathscr{J}$-irreducible}
if $M\backslash\{0\}$ has exactly one minimal $G\times G$-orbit. 
Equivalently, there is a unique minimal nonzero idempotent  
$e_1\in E(\overline{T})$ 
such that $\Lambda_1=\{e_1\}$. 
It follows that 
$E_1\simeq W/C_W(e_1)$, and 
$fe_1=e_1$ for all $f\in \Lambda\setminus \{0\}$ (Subsection 1.1).    
Clearly, simple projective group embeddings 
are exactly the projectivizations of $\mathscr{J}$-irreducible monoids (Subsection 1.2). 
See  \cite{pr:idem} 
or \cite[Section 7.3]{re:lam} for a systematic discussion of $\mathscr{J}$-irreducible monoids. 
Due to the following result, 
any $\mathscr{J}$-irreducible monoid is semisimple and 
can be obtained via the procedure of Example \ref{construction.ex}.

\begin{thm}[\protect{\cite{pr:idem}}] \label{jirred.thm}
A reductive monoid $M$, with zero, is 
	$\mathscr{J}$-irreducible 
if and only if there is an irreducible rational representation
	      $\rho : M\to End(V)$ which is finite as a morphism
	      of algebraic varieties. \qed
\end{thm}

Let $M$ be a $\mathscr{J}$-irreducible monoid with $\Lambda_1=\{e_1\}$, as above.
We say that
$M$ is {\em $\mathscr{J}$-irreducible of type $J$} if 
$
J=\{s\in \Sigma\;|\; se_1=e_1s\},
$ 
where $\Sigma$ is the set of simple reflections of $W$. 
Notice that $C_W(e_1)=W_J$, the subgroup of $W$ generated by $J$.  
The set $J$ can be determined in terms of any irreducible representation
satisfying 
Theorem \ref{jirred.thm}. 
It is worthwhile to pause and notice that $\Lambda$ is completely determined by $J$. 

\begin{thm}[\protect{\cite{pr:idem}}]\label{orbits.jirred.thm} 
If $M$ is a $\mathscr{J}$-irreducible monoid of type $J$, then  
$$\Lambda\setminus \{0\}\simeq \{I\subseteq \Sigma\;|\;\textrm{no connected component of}\;\, I\; \textrm{is contained entirely in}\; J\},$$ 
in such a way that $e$ corresponds to $I\subseteq \Sigma$ if $I=\{s\in \Sigma\,|\,se=es\neq e\}$. This bijection 
identifies $\Lambda_2$ with $\Sigma \setminus J$. \qed 
\end{thm}

\smallskip 

Let $X=\P(M)$ be a simple projective embedding, where 
$M$ is a $\mathscr{J}$-irreducible monoid, $\Lambda_1=\{e_1\}$, and $J=\{s\in \Sigma \,|\, se_1=e_1s\}$.   
Then 
$P_{e_1}=\bigsqcup_{w\in W_J} BwB=P_J,$   
where 
$P_J\subset G$ is the standard 
parabolic subgroup associated to $J$. 
Hence, by Theorem \ref{orbits.thm}, 
the unique closed orbit of $X$ is 
$G[e_1]G\simeq G/P_J\times G/P_J^-$. 
Note that $X$ is toroidal only when $J=\emptyset$ (Proposition \ref{quasireg.cor}). 
In general, the $G\times G$-orbit structure of $X=\P(M)$ is completely determined by $J$ (Theorem \ref{orbits.jirred.thm}). 

\begin{dfn}
Let $X$ be a simple projective group embedding. 
We say that $X$ is {\bf simple of type $J$}
if $X=\P(M)$, where $M$ is a $\mathscr{J}$-irreducible monoid
of type $J$. 
\end{dfn}

The type of a simple group embedding 
is independent of its 
presentation as a projectivization of a monoid \cite{re:hpolyirr}. 
Renner \cite{re:desc,re:hpolyirr} has classified all 
rationally smooth simple group embeddings according to their type. Below is the list. 

\begin{thm}[\protect{\cite[Proposition 2.8]{re:hpolyirr}}] \label{class.thm}
For each irreducible Dynkin diagram 
the following is a list of all types, $J\subset \Sigma$, 
of $\mathscr{J}$-irreducible monoids $M$ of type $J$ 
such that $\P(M)$ is rationally smooth. 
%
The numbering of the elements of $\Sigma$ is as follows. For types $A_n, B_n, C_n, F_4,$ and
$G_2$ it is the usual numbering. In these cases the end nodes are $s_1$ and $s_n$. For type
$E_6$ the end nodes are $s_1,s_5$ and $s_6$ with $s_3s_6\neq s_6s_3$. For type
$E_7$ the end nodes are $s_1,s_6$ and $s_7$ with $s_4s_7\neq s_7s_4$. For type
$E_8$ the end nodes are $s_1,s_7$ and $s_8$ with $s_5s_8\neq s_8s_5$.
In each case  of type $E_n$, the nodes corresponding to $s_1, s_2,...,s_{n-1}$
determine the unique subdiagram of type $A_{n-1}$. For type $D_n$ the end
nodes are $s_1,s_{n-1}$ and $s_n$. The two subdiagrams of $D_n$, of type $A_{n-1}$,
correspond to the subsets $\{s_1, s_2,...,s_{n-2},s_{n-1}\}$ and
$\{s_1, s_2,...,s_{n-2},s_n\}$ of $\Sigma$. For $s_i\in \Sigma$, the 
corresponding simple root is denoted 
$\alpha_i$. 
\begin{enumerate}
   \item     $A_n$, $n\geq 1$. Let $\Sigma=\{s_1,...,s_n\}$.
\begin{enumerate}[(a)]
        \item $J=\phi$.
	\item $J=\{s_1,...,s_i\}$, $1\leq i<n$.
	\item $J=\{s_j,...,s_n\}$, $1<j\leq n$.
	\item $J=\{s_1,...,s_i,s_j,...s_n\}$, $1\leq i$, $i\leq j-3$ and $j\leq n$.
\end{enumerate}
	 \item     $B_n$, $n\geq 2$. Let $\Sigma=\{s_1,...,s_n\}$, $\alpha_n$ short.
\begin{enumerate}[(a)]
  \item $J=\phi$.
	\item $J=\{s_1,...,s_i\}$, $1\leq i<n$.
	\item $J=\{s_n\}$.
	\item $J=\{s_1,...,s_i,s_n\}$, $1\leq i$ and $i\leq n-3$.
\end{enumerate}
   \item $C_n$, $n\geq 3$. Let $\Sigma=\{s_1,...,s_n\}$, $\alpha_n$ long.
\begin{enumerate}[(a)]
  \item $J=\phi$.
	\item $J=\{s_1,...,s_i\}$, $1\leq i<n$.
	\item $J=\{s_n\}$.
	\item $J=\{s_1,...,s_i,s_n\}$, $1\leq i$ and $i\leq n-3$.
\end{enumerate}
	\item $D_n$, $n\geq 4$. Let $\Sigma=\{s_1,...s_{n-2},s_{n-1},s_n\}$.
\begin{enumerate}[(a)]
	\item $J=\phi$.
	\item $J=\{s_1,...,s_i\}$, $1\leq i\leq n-3$.
	\item $J=\{s_{n-1}\}$.
	\item $J=\{s_n\}$.
	\item $J=\{s_1,...,s_i,s_{n-1}\}$, $1\leq i\leq n-4$.
	\item $J=\{s_1,...,s_i,s_n\}$, $1\leq i\leq n-4$.
\end{enumerate}
	\item $E_6$. Let $\Sigma=\{s_1,s_2,s_3,s_4,s_5,s_6\}$.
\begin{enumerate}[(a)]
	\item $J=\phi$.
	\item $J=\{s_1\}$ or $\{s_1,s_2\}$.
	\item $J=\{s_5\}$ or $\{s_4,s_5\}$.
	\item $J=\{s_6\}$.
  \item $J=\{s_1,s_5\},\{s_1,s_2,s_5\}$ or $\{s_1,s_4,s_5\}$.
  \item $J=\{s_1,s_6\}$.
  \item $J=\{s_5,s_6\}$
  \item $J=\{s_1,s_5,s_6\}$.
\end{enumerate}
	\item $E_7$. Let $\Sigma=\{s_1,s_2,s_3,s_4,s_5,s_6,s_7\}$.
\begin{enumerate}[(a)]
  \item $J=\phi$.
	\item $J=\{s_1\}, \{s_1,s_2\}$ or $\{s_1,s_2,s_3\}$.
	\item $J=\{s_6\}$ or $\{s_5,s_6\}$.
	\item $J=\{s_7\}$.
  \item $J=\{s_1,s_6\},\{s_1,s_2,s_6\},\{s_1,s_2,s_3,s_6\},\{s_1,s_5,s_6\},$
  or $\{s_1,s_2,s_5,s_6\}$.
  \item $J=\{s_6,s_7\}$.
  \item $J=\{s_1,s_7\}$ or $\{s_1,s_2,s_7\}$.
  \item $J=\{s_1,s_6,s_7\},\{s_1,s_2,s_6,s_7\}$.
\end{enumerate}
	\item $E_8$. Let $\Sigma=\{s_1,s_2,s_3,s_4,s_5,s_6,s_7,s_8\}$.
\begin{enumerate}[(a)]
  \item $J=\phi$.
	\item $J=\{s_1\}, \{s_1,s_2\}, \{s_1,s_2,s_3\}$ or $\{s_1,s_2,s_3,s_4\}$.
	\item $J=\{s_7\}$ or $\{s_6,s_7\}$.
	\item $J=\{s_8\}$.
  \item $J=\{s_1,s_7\},\{s_1,s_2,s_7\},\{s_1,s_2,s_3,s_7\}, \{s_1,s_2,s_3,s_4,s_7\}$,\\
  $\{s_1,s_6,s_7\},\{s_1,s_2,s_6,s_7\}, \{s_1,s_2,s_3,s_6,s_7\}$
  or $\{s_1,s_2,s_5,s_6\}$.
  \item $J=\{s_7,s_8\}$.
  \item $J=\{s_1,s_8\}, \{s_1,s_2,s_8\}$ or $\{s_1,s_2,s_3,s_8\}$.
  \item $J=\{s_1,s_7,s_8\},\{s_1,s_2,s_7,s_8\}$.
\end{enumerate}
	\item $F_4$. Let $\Sigma=\{s_1,s_2,s_3,s_4\}$.
\begin{enumerate}[(a)]
	\item $J=\phi$.
	\item $J=\{s_1\}$ or $\{s_1,s_2\}$.
	\item $J=\{s_4\}$ or $\{s_3,s_4\}$.
	\item $J=\{s_1,s_4\}$.
\end{enumerate}
	\item $G_2$. {\em Let} $\Sigma=\{s_1,s_2\}$.
\begin{enumerate}[(a)]
	\item $J=\phi$.
	\item $J=\{s_1\}$.
	\item $J=\{s_2\}$.
\end{enumerate} \qed
\end{enumerate} 
\end{thm}

\smallskip

According to this list, if $G$ is a semisimple group of adjoint type, then 
the choice $J=\emptyset$ yields the wonderful compactification of $G$. 
On the other hand, in type $A_n$, the choice $J=\{s_2,\ldots,s_n\}$ 
yields $\P^{(n+1)^2-1}$, a compactification of $PGL_{n+1}$.   

\begin{rem} 
Projective group embeddings are also described in terms of weight polytopes \cite{ti:emb}. 
%
It is possible to state Theorem \ref{class.thm} 
in that language as well. Indeed, 
%
let $M$ be a $\mathscr{J}$-irreducible monoid. Now let $\rho:M\to End(V)$ 
be an irreducible representation as in Theorem \ref{jirred.thm}.  
Let $G_0$ be the semisimple part of $G$, with maximal torus $T_0=G_0\cap T$, and let 
$\rho_\lambda=\rho|G_0$ with highest weight $\lambda \in \mathcal{C}$, the rational Weyl chamber of $G_0$. 
One checks that $J=\{s\in \Sigma\,|\, s(\lambda)=\lambda\}$ (see e.g. \cite{re:desc}). 
So we can identify $e_1$ with $\lambda$.  
Now the {\em weight polytope} $\mathcal{P}_\lambda$ 
is defined to be  
the convex hull of $W\cdot \lambda$ in $X(T_0)\otimes \Q$, where $X(T_0)$ 
is the character group of $T_0$. 
This $W$-invariant polytope is considered in \cite{ti:emb}. 
Observe that 
$\mathcal{P}_\lambda$ is the polytope associated to the 
toric variety $\P(\overline{T})$. In particular, $E_1\simeq W/C_W(e_1)$ corresponds 
to the set of vertices of $\mathcal{P}_\lambda$.  
But even more is true. By \cite[Corollary 2.3]{re:hpolyirr} the face lattice $\mathcal{F}_\lambda$  
of $\mathcal{P}_\lambda$ is completely described in terms of 
the Weyl group $(W,\Sigma)$. 
Indeed, the set of $W$-orbits of $\mathcal{F}_\lambda$ is in one-to-one 
correspondance with $\{I\subseteq \Sigma\,|\,\textrm{no connected component of}\; I\; \textrm{is contained entirely in}\; J\}$. 
The latter is $\Lambda \setminus\{0\}$, by Theorem \ref{orbits.jirred.thm}. 
The correspondance assigns to $I\subseteq \Sigma$ the unique face $F\in \mathcal{F}_\lambda$ with 
$I=\{s\in \Sigma\,|\,s(F)=F \,\textrm{and}\, s|F\neq id\}$ whose relative interior $F^0$ has 
nonempty intersection with $\mathcal{C}$. 
It follows from Theorem \ref{ratsm.thm} that 
$\P(M)$ is rationally smooth  
if and only if  $\mathcal{P}_\lambda$ is a simple polytope, 
i.e. there are exactly $|\Sigma|$ edges of $\mathcal{P}_\lambda$ meeting at $\lambda$. 
Observe that 
$\Sigma \setminus J$ 
corresponds to 
the set of fundamental dominant
weights involved in the description of $\rho_\lambda$; 
that is, $\lambda$ is a dominant weight of the form  
$\lambda=\sum m_i\lambda_i$, where   
$\lambda_i$ runs over  
the fundamental weights attached to $s_i\in \Sigma\setminus J$, 
and $m_i$ are positive (nonzero) numbers. 
For group embeddings with more than one closed orbit there is certainly 
a dictionary between Timashev's description and 
Renner's. 
The interested reader should consult 
\cite[Section 7.2]{re:lam}, \cite{re:hpolyirr} and \cite{re:desc}. 
We shall not pursue this here. 
 \end{rem}

\subsection{Equivariant cohomology of simple group embeddings}  

Let $X=\P(M)$ be a simple embedding of type $J$.
Given that $X$ has only one closed orbit, we can associate to 
any $f\in E_2$ a unique reflection $s_{\alpha_f}$ such that 
$s_{\alpha_f}f=fs_{\alpha_f}\neq f$ (Lemma \ref{hclass.lem} and \ref{thetaforex3.lem}). 
Let $\Lambda_1=\{e_1\}$ and put $$L^J = \{f\in E_2\,|\, fe_1=e_1\}.$$ 

\begin{thm}\label{mainj.thm}
Let $X=\mathbb{P}(M)$ be a simple embedding of type $J$.
If $X$ is rationally smooth, then the natural morphism $$H^*_{T\times T}(X)\to H^*_{T\times T}(G/P_J\times G/P_J^-)$$ is injective, 
and its image consists of all maps $\varphi\in H^*_{T\times T}(G/P_J\times G/P_J^-),$ 
subject to the condition: 
for every 
$f\in L^J$, $u\in W$, and $v\in W$,  
the following holds 
$$\varphi(u\,e_1\,v)\equiv \varphi(u\,s_{\alpha_f}\,e_1\, s_{\alpha_f}\, v) \;
{\rm mod} \;(\alpha_f\circ int(u^{-1}), -\alpha_f \circ int(v)),$$
where $\alpha_f$ is the root associated 
to the reflection $s_{\alpha_f}$. 
\end{thm}

\begin{proof}
The first assertion is a direct consequence of Theorem \ref{main.thm}. 
Besides, $X$ contains a unique closed $G\times G$-orbit,  
so to describe the image  
we just need to focus 
on translating condition (1) of Theorem \ref{main.thm} into our situation. 
Let $f\in E_2$. 
Then there are exactly two rank-one idempotents $f_1,f_2$, such that
$f_1f=f_1$, $f_2f=f_2$ and $f_2=s_\alpha f_1 s_\alpha$, where
$s_\alpha f=s_\alpha f\neq f$.
On the other hand, because $\Lambda_1=\{e_1\}$, 
then $f_1=ue_1u^{-1}$, for some $u\in W$.
The latter implies that $f'=u^{-1}fu$ is an idempotent of $\overline{T}$ such that 
$f'e_1=e_1$, 
that is, $f'\in L^J$.
In short, any $f \in E_2$ 
is conjugate to an element of $L^J$. 
This observation and Theorem \ref{main.thm} (1) yield the result.   
\end{proof}

\begin{cor}
Let $X=\mathbb{P}(M)$ be a rationally smooth simple embedding of type $J$.
Let $e_1$ be the unique rank-one idempotent for which $\Lambda_1=\{e_1\}$. 
Then the ring $H^*_{G\times G}(X)$ consists of all $\varphi \in (S\otimes S)^{W_J \times W_J}$
such that 
$$(s_{\alpha_f},s_{\alpha_f})\,\varphi \equiv \varphi \; {\rm mod} \;(\alpha_f, -\alpha_f ),$$
for every $f\in \Lambda_2\subset L^J$.
 \end{cor}

\begin{proof}
Simply translate Corollary \ref{gequiv.thm} 
into this situation, 
using Theorem \ref{mainj.thm}. 
\end{proof}


\end{document}